\documentclass[10pt]{amsart}
\usepackage{amssymb}
\usepackage{mathrsfs}
\usepackage{hyperref}
\usepackage{color}

\newtheorem{Theorem}{Theorem}[section]
\newtheorem{Definition}{Definition}[section]
\newtheorem{Proposition}[Theorem]{Proposition}
\newtheorem{Lemma}[Theorem]{Lemma}
\newtheorem{Corollary}[Theorem]{Corollary}
\newtheorem{Claim}[Theorem]{Claim}
\newtheorem{Remark}[Theorem]{Remark}

\newtheorem{theoremalph}{Theorem}
\newtheorem{corollaryalph}{Corollary}

\def\DD{{\mathbb D}}

\def\LL{{\mathbb L}}
 \def\NN{{\mathbb N}} 
\def\PP{{\mathbb P}}
 \def\RR{{\mathbb R}}

 \def\ZZ{{\mathbb Z}}

\def\cA{{\mathcal A}}    
\def\cB{{\mathcal B}}  \def\cH{{\mathcal H}} \def\cN{{\mathcal N}} 
    
   \def\cP{{\mathcal P}} 
  \def\cK{{\mathcal K}}  
\def\cF{{\mathcal F}}  \def\cL{{\mathcal L}}

\begin{document}

\title[SRB measures for some diffeomorphisms]{SRB measures for some diffeomorphisms with dominated splittings as zero noise limits}

\author{Zeya Mi}
%\small\it Department of mathematics, Soochow University, Suzhou 215006, Jiangsu, P.R.China}
\date{\today}

\begin{abstract}
In this paper, we provide a technique result on the existence of Gibbs $cu$-states for diffeomorphisms with dominated splittings. More precisely, for given $C^2$ diffeomorphim $f$ with dominated splitting $T_{\Lambda}M=E\oplus F$ on an attractor $\Lambda$, by considering some suitable random perturbation of $f$, we show that for any zero noise limit of ergodic stationary measures, if it has positive integrable Lyapunov exponents along invariant sub-bundle $E$, then its ergodic components contain Gibbs $cu$-states associated to $E$. With this technique, we show the existence of SRB measures and physical measures for some systems exhibitting dominated splittings and weak hyperbolicity.
\end{abstract}
\thanks{Z. Mi was partially supported by NSFC 11801278 and The Startup Foundation for Introducing Talent of NUIST(Grant No. 2017r070)}

\maketitle

%\tableofcontents
\section{Introduction and main results}\label{se1}
Our aim of this paper is to show the existence of SRB (or physical) measures for some systems beyond hyperbolicity. The existence is based on a criterion established to get Gibbs $cu$-sates by techniques of adding random perturbations.

SRB measures were discovered by Sinai, Ruelle and Bowen in 1970s for uniformly hyperbolic systems \cite{Sin72,BoR75,Bow75,Rue76}.
After that, the SRB theory rapidly became a central topic of dynamical systems. Beyond the uniform hyperbolicity, up to now, there is no systematic approach for the construction of SRB measures, it remains to be a challenging problem in the ergodic theory of dynamical systems(see \cite{v2,pa} for related conjectures). An invariant Borel probability is said to be an SRB measure if it has positive Lyapunov exponents almost everywhere and its disintegration along Pesin unstable manifolds are absolutely continuous w.r.t. Lebesgue measures. SRB measures can also charactered by the fact they have positive Lyapunov exponents and satisfy the Pesin entropy formula \cite{LeY85}. Sometimes, SRB measures are expect to be physical observable \cite{ER}, known as \emph{physical measures},  i.e., the sets of generic points have positive Lebesgue measure on the manifold (e.g. see \cite{y02} for a survey).

As for the construction of SRB (or physical) measures one should refer the Gibbs $u$-states and Gibbs $cu$-states. Gibbs $u$-states were introduced and studied by Pesin and Sinai \cite{ps82} for partially hyperbolic diffeomorphisms admitting strong unstable sub-bundles, which requires the absolute continuity of conditional measures along strong unstable manifolds(see \cite[Chapter 11]{bdv} for more details). In \cite{ABV00}, Alves, Bonatti and Viana proposed the notion of Gibbs $cu$-states (see Definition \ref{DCU}) for diffeomorphisms with dominated splittings as a nonuniform notion of Gibbs $u$-states. It turns out that Gibbs $u$-states and Gibbs $cu$-states are crucial candidates of SRB measures and physical measures(e.g. \cite{ABV00, bv, CY05, AV15, MCY17}).

In this paper, we deal with the problem of the existence of SRB measures for some systems exhibitting dominated splitting. The key idea is to add small random noise to the deterministic dynamical system and prove that as noise levels tends to zero, the limit of the \emph{ergodic} stationary measures, called the \emph{randomly ergodic limit}(see Definition \ref{randomlylimit}), has ergodic components to be Gibbs $cu$-states associated to some sub-bundle $E$ whenever this randomly ergodic limit appears some weak expansion along $E$.  Under some extra assumptions on the other directions of the sub-bundles, one can make these Gibbs $cu$-states to be SRB measures or physical measures.

Let $f$ be a $C^2$ diffeomorphism on a compact Riemannian manifold $M$. Let $\Lambda$ be an \emph{attractor} for $f$ with the \emph{trapping region} $U$, i.e., $\Lambda$ is a compact $f$-invariant subset of $M$, and $U$ is the open neighborhood of $\Lambda$ such that
$f(\overline{U})\subset U$ and $\Lambda=\cap_{n\ge 0}f^n(U)$.
We say $\Lambda$ admits a \emph{dominated splitting}, if there exists a $Df$-invariant splitting $T_{\Lambda}M=E\oplus F $ and constant $0<\tau<1$ such that
$\|Df|_{F(x)}\|\le \tau \|Df^{-1}|_{E(f(x))}\|$ for any $x\in \Lambda$.
%$$
%\frac{\|Df|E(x)\|}{\|Df^{-1}|F(f(x))\|}\le \tau~~~~~ \textrm{for any}~ x\in \Lambda.
%$$
In this case, we say $E$ dominates $F$, and denote it by $E\oplus_{\succ} F$.

\begin{theoremalph}\label{m2}
Let $f$ be a $C^2$ diffeomorphism with an attractor $\Lambda$ admitting a dominated splitting $T_{\Lambda} M=E\oplus_{\succ} F$. If $\mu$ is a randomly ergodic limit supported on $\Lambda$ such that
\begin{itemize}
\item $\displaystyle{\int \log \|Df^{-1}|_E\|^{-1}{\rm d}\mu>0}$;
\bigskip
\item $\lim_{n\to +\infty}\frac{1}{n}\log \|Df^{n}|_{F(x)}\|\le 0$ for $\mu$-almost every $x$.
\end{itemize}
Then there exist ergodic components of $\mu$ as SRB measures.
\end{theoremalph}

By adjusting the condition on $F$-direction one can deduce the following corollary.
\begin{corollaryalph}\label{cp}
Under the assumption of Theorem \ref{m2}, if we have
$$\lim_{n\to +\infty}\frac{1}{n}\log \|Df^{n}|_{F(x)}\|< 0$$
for $\mu$-almost every $x$,
then there exist physical measures in the ergodic components of $\mu$.
\end{corollaryalph}

The existence of SRB measures and physical measures obtained above mainly based the following tool, which focus on the existence of Gibbs $cu$-states by using the method of adding random perturbations.

\begin{theoremalph}\label{main}
Let $f$ be a $C^2$ diffeomorphism with an attractor $\Lambda$ admitting a dominated splitting $T_{\Lambda} M=E\oplus F$. If $\mu$ is a randomly ergodic limit supported on $\Lambda$ such that
$$
\int \log \|Df^{-1}|_E\|^{-1}{\rm d}\mu>0,
$$
then there exist ergodic components of $\mu$ as Gibbs cu-states associated to $E$.
\end{theoremalph}

In the previous work of Cowieson-Young \cite{CY05}, they showed the existence of SRB measures as zero noise limits in partially hyperbolic attractors with one-dimensional centers. The underlying principle there is that if the zero noise limit  (accumulation point of stationary measures as noise levels tends zero) has the no mixed behavior and the system is random entropy expansiveness, then it satisfies the Pesin entropy formula by passing limit of the random Pesin entropy formula.
With the similar token, but using asymptotically entropy expansiveness property instead, Liu-Lu \cite{LiL15} obtained SRB measures by showing Pesin entropy formula holds for every zero noise limit for some partially hyperbolic attractors. Note that without using the tool of random perturbations, \cite{CY} got the existence of invariant measures satisfying Pesin entropy formula for more general dynamical systems (see \cite{ZYC} also).

In contrast to their strategy, in the proof of Theorem \ref{main} or Theorem \ref{m2}, we do not involve the discussion on Pesin entropy formula, and different to them, we consider special zero noise limits---randomly ergodic limits, and take advantage of this modification, as \emph{ergodic} stationary measures inherit more information from their limit than stationary measures. Indeed, as a central step, we show that under the assumptions of our theorems on $E$-direction, one can get some uniform behaviors with respect to these ergodic stationary measures, which guarantee that the randomly ergodic limit admits the absolute continuity property along a family of unstable manifolds tangent to $E$ (see Theorem \ref{bac} in $\S$\ref{se4}).

%we do not need to discuss the Pesin entropy formula for zero noise limits, but select Gibbs cu-states from the ergodic components of the zero noise limits directly.
%
%
%More precisely, in terms of the simple dichotomy of central Lyapunov exponents caused by dimension assumption, one gets that for any accumulation point of stationary measures as noise levels tends zero,
%
%
%showed the existence of SRB measures for partially hyperbolic diffeomrophisms with one-dimensional centers. Liu-Lu \cite{LiL15} show the existence of SRB measures for attractors having two invariant sub-bundles, where one is uniformly expanding and the other one has non-positive Lyapunov exponents everywhere. The key step of both works is to show the existence of Pesin entropy formula for zero noise limits. In contrast to their strategy, we do not need to discuss the Pesin entropy formula for zero noise limits, but select Gibbs cu-states from the ergodic components of the zero noise limits directly. Indeed, as a byproduct of our approach, one can also establish the existence of SRB measures for systems considered in  \cite{CY05,LiL15}, see Section \ref{se6} for detalis.
%
In recently, \cite{AV15, MCY17} showed the existence (and finiteness) of SRB/physical measures on mostly expanding diffeomorphisms. Recall that in both works, the authors make effort to show the existence of positive Lebesgue measure set of (weakly) non-uniformly expanding points, then in terms of the previous techniques from \cite{ADLP17} or \cite{ABV00} to find SRB measures. In $\S$\ref{se6}, as a byproduct of our result, we provide a proof of the existence of SRB measures for systems considered in \cite{MCY17} (it works for \cite{AV15} with simpler arguments). We would like to mention that by our new arguments the partially hyperbolic splittings can be restricted to attractors, rather than the whole manifold, which is pivotal there to find the (weakly) non-uniformly expanding points.

The remainder of this paper is organized as follows. In $\S$\ref{se2}, we recall the definitions of SRB measures and Gibbs $cu$-states. In $\S$\ref{se3}, we study the dynamics under \emph{regular} random perturbations. $\S$\ref{se4} is dedicated to show that any randomly ergodic limit of
Theorem \ref{main} has the absolute continuity property along a family of local unstable disks. In $\S$\ref{se5}, we prove a technique result which asserts the arising of Gibbs $cu$-states from ergodic components of some invariant measures. Finally, the proofs of Theorem \ref{m2} and Theorem \ref{main} are provided in $\S$\ref{se6}, in which some applications of them are also studied.

\bigskip
{\bf{Acknowledgements.}} The author is greatly appreciated for the discussion
with Professor Dawei Yang.

\section{SRB measures and Gibbs cu-states}\label{se2}

Let $M$ be a compact Riemannian manifold, use ${\rm Leb}$ represent the Lebesgue measure of $M$. Given a sub-manifold $\gamma\subset M$, denote by ${\rm Leb}_{\gamma}$ the Lebesgue measure on $\gamma$ induced by the restriction of the Riemannian structure to $\gamma$. Let $d$ denote the distance in $M$, and $\rho$ the distance in the Grassmannian bundle of $TM$ generated by the Riemannian metric. Denote by ${\rm Diff}^2(M)$ the space of $C^2$ diffeomorphisms on $M$.

Given diffeomorphism $f$ on $M$, let $E$ be a $Df$-invariant sub-bundle and $\mu$ an $f$-invariant measure, we define the \emph{integrable Lyapunov exponent along $E$} of $\mu$ by
$$\lambda_E(\mu)=\int \log \|Df^{-1}|_E\|^{-1}{\rm d}\mu.$$
%Let $\Lambda$ be an attractor for $f$ with the trapping region $U$ in the sense that $\Lambda$ is an compact invariant subset of $M$, and $U$ is the open neighborhood of $\Lambda$ such that
%$f(\overline{U})\subset U$ and $\Lambda=\cap_{n\ge 0}f^n(U)$.
%
%We say $\Lambda$ admits a \emph{dominated splitting}, if there exists a $Df$-invariant splitting $T_{\Lambda}M=E\oplus_{\succ} F $ and constant $0<\gamma<1$ such that
%$$
%\frac{\|Df|E(x)\|}{\|Df^{-1}|F(f(x))\|}\le \gamma~~~~~ \textrm{for any}~ x\in \Lambda.
%$$
%If we have the splitting as above, we say that $E$ dominates $F$, and we denote it by $E\oplus_{\succ} F$.
%
%In particular, if $E=E^u$ or $F=E^s$, then we say the dominated splitting as above is a partially hyperbolic splitting.

Given $f\in {\rm Diff}^2(M)$, let $\Lambda$ be an attractor admitting the dominated splitting $T_{\Lambda}M=E\oplus_{\succ} F$. Since the distributions $E$ and $F$ are continuous, we may extend them continuously to some trapping region $U$ of $\Lambda$, denoted by $E$ and $F$ as well. Given $a>0$, define the $E$-direction \emph{cone field}
$\mathcal{E}_a=\{\mathcal{E}_a(x): x\in U\}$ of width $a$ by
$$
\mathcal{E}_a(x)=\left\{\nu=\nu_E+\nu_F \in E(x)\oplus F(x): \|v_F\|\le a\|v_E\|\right\}
$$
for every $x\in U$.
We say a smooth embedded sub-manifold $D$ is \emph{tangent to} $\mathcal{E}_a$ if ${\rm dim}D={\rm dim}E$ and $T_xD\subset \mathcal{E}_a(x)$ for any $x\in D$.

\smallskip
%In partially hyperbolic system admitting $E^u$, a \emph{Gibbs u-state} is an $f$-invariant Borel probability whose disintegration along strong unstable manifolds are absolutely continuous w.r.t. Lebesgue measures on these unstable manifolds. See \cite{ps82} and \cite{bdv} for its existence and compactness, etc. As a non-uniform version of Gibbs u-states, Alves, Bonatti, Viana \cite{ABV00} proposed the notion of \emph{Gibbs cu-states} for diffeomorphisms with dominated splttings. Gibbs u-states and Gibbs cu-states are crucial candidates of \emph{SRB} (Sinai-Ruelle-Bowen) measures, whose definitions will be recalled later.

Let $(X,\cA, \mu)$ be a probability space, we say $\cP$ is a \emph{measurable partition} of $X$, when there exists a sequence of countable partitions $\{\cP_k: k\in \NN\}$ of $X$ such that $\cP=\bigvee_{k=0}^{\infty}\cP_k$ (mod 0).
%Let $\pi:X\to \cP$ be the projection which maps each point of $X$ to the element of $\cP$ containing it. Denote by $\hat{\mu}$ the quotient measure of $\mu$ with respect to $\cP$ defined by $\hat{\mu}(A)=\mu(\pi^{-1}A)$ for every measurable subset $A$ of $\cP$.
By Rokhlin's disintegration theorem \cite{R62} (see \cite[Appendix C.4]{bdv} also), for measurable partition $\cP$, there exists a unique family of \emph{conditional measures} $\{\mu_{P}: P \in \cP\}$ of $\mu$ w.r.t. $\cP$ such that $\mu_{P}(P)=1$ for $\hat{\mu}$-almost every $P\in \cP$, and for any measurable set $A$, $P\mapsto \mu_{P}(A)$ is measurable with $\mu(A)=\int \mu_P(A){\rm d}\hat{\mu}$, where $\hat{\mu}$ is the quotient measure of $\mu$ w.r.t. $\cP$.
%\begin{itemize}
%\item $\mu_{P}(P)=1$, for $\hat{\mu}$-almost every $P\in \cP$;
%\item for any bounded measurable function $\phi$ on $X$, we have $P\mapsto \int \phi {\rm d}\mu_{P}$ is measurable, and
%$$
%\int \phi {\rm d}\mu=\int \int \phi {\rm d}\mu_{P}{\rm d}\hat{\mu}.
%$$
%\end{itemize}

%Let $E$ be a $Df$-invariant bundle for diffeomorphism $f$.
%We represent $W^{E,u}_{\delta}(x)$ as the local unstable disk of $x$ with radius $\delta$ around $x$, which is tangent to $E$ and $(C,\lambda)$-\emph{backward contracting} (for some $C>0$ and $\lambda\in (0,1)$), i.e.,
%$$
%d(f^{-n}(y), f^{-n}(z))\le C\lambda^n d(y,z)\quad \textrm{for every}~n\in \NN,
%$$
%whenever $y,z \in W_{\delta}^{E,u}(x)$.

\begin{Definition}\label{def:abs}
Let $\mu$ be a Borel probability and $\cP$ a measurable partition formed by disjoint smooth sub-manifolds, whose union admits positive $\mu$ measure. We say $\mu$ has \emph{absolutely continuous conditional measures} along $\cP$, if the conditional measures $\{\mu_{\gamma}:\gamma \in \cP\}$ of $\mu$ (restricted to $\cup_{\gamma\in \cP}\gamma$) w.r.t. $\cP$ are absolutely continuous w.r.t. corresponding Lebesgue measures on these sub-manifolds, i.e., $\mu_{\gamma}\ll {\rm Leb}_{\gamma}$\footnote{We write $``\mu_1 \ll \mu_2 "$ to denote that $\mu_1$ is absolutely continuous w.r.t. $\mu_2$}  for $\hat{\mu}$-almost every $\gamma\in \cP$.
\end{Definition}

Let $\mu$ be an $f$-invariant Borel probability that has positive Lyaunov exponents almost everywhere, then there exist a family of Pesin unstable manifolds $W^u(x)$ for $\mu$-almost every $x$ (see \cite{bp02} for Pesin theory). A measurable partition $\xi$ is said to be \emph{subordinate to} $W^u$ if for $\mu$-almost every $x$,
\begin{itemize}
\item $\xi(x)\subset W^u(x)$, where $\xi(x)$ is the element of $\xi$ containing $x$,
\item $\xi(x)$ contains an open neighborhood of $x$ in $W^u(x)$.
\end{itemize}
In particular, for diffeomorphism $f$ with dominated splitting of type $E\oplus_{\succ} F$, if $\mu$ is a Borel probability that has positive Lyapunov exponents along $E$, then there exists a Pesin unstable manifold $W^{E,u}(x)$ tangent to $E$ at $x$ with dimension ${\rm dim}E$, for
$\mu$-almost every $x$. Thus, measurable partitions subordinate to $W^{E,u}$ can be defined similarly.

Now we recall the definition of SRB measures, following \cite{y02,ER}:

\begin{Definition}\label{Def:SRB}
Let $f$ be a $C^2$ diffeomorphism on $M$, an $f$-invariant Borel probability $\mu$ is an SRB measure if
$\mu$ has positive Lyapunov exponents almost everywhere, and it has absolutely continuous conditional measures along any measurable partitions subordinate to $W^u$.
\end{Definition}

\begin{Definition}\label{DCU}
Assume that $f$ is a $C^2$ diffeomorphism with an attractor $\Lambda$ admitting a dominated splitting $T_{\Lambda} M=E\oplus F$.

An $f$-invariant Borel probability $\mu$ supported on $\Lambda$ is called a Gibbs cu-state (associated to $E$) if the Lyapunov exponents of $\mu$ along  $E$ are positive and
$\mu$ has absolutely continuous conditional measures along any measurable partition subordinate to $W^{E,u}$.
\end{Definition}

We remark that to check the absolute continuity property for SRB measures and Gibbs $cu$-states, it suffices to verify it for one (not all) measurable partition subordinate to Pesin unstable manifolds(see e.g. \cite[Chapter IV: Remark 2.1]{LiQ95}).

%One can also define the absolutely continuity property in local sense, where measurable partitions are formed by uniform local unstable manifolds at Pesin blocks.

%For any measurable partition formed by unstable manifolds, we use the following terminology for convenience.
%
%\begin{Definition}\label{def:abs}
%Given a Borel probability $\mu$,
%let $\cP$ be a measurable partition formed by disjoint local unstable disks, whose union admits positive $\mu$ measure. We say $\mu$ has \emph{absolutely continuous conditional measures} along $\cP$, if $\mu_{\gamma}\ll {\rm Leb}_{\gamma}$ for $\hat{\mu}$ almost every $\gamma\in \cP$.
%\end{Definition}

%\begin{Definition}\label{Gibbscu}
%If $f\in {\rm Diff}^2(M)$ has an attractor $\Lambda$ with dominated splitting $T_{\Lambda}M=E\oplus_{\succ} F$. We say an $f$-invariant measure $\mu$ supported on $\Lambda$ is a Gibbs $E$-state if
%\begin{itemize}
%\item[1)]
%${\rm dim}E$ larger Lyapunov exponents are positive $\mu$-almost everywhere;
%\item[2)] the conditional measures of $\mu$  are absolutely continuous  with respect to Lebesgue measures along local unstable manifolds tangent to $E$.
%\end{itemize}
%\end{Definition}

\section{Random perturbations}\label{se3}
\subsection{Regular random perturbations}

In this work, we use the random perturbation model of iterations of random maps, the strategy is to consider random orbits generated by an independent and identically distributed random choice of map at each iteration (see \cite{v97}, \cite{Alv03},\cite[Appendix D.2]{bdv}).
Given $f\in {\rm Diff}^2(M)$,  take a metric space $\Omega$ and define the continuous map $\mathscr{F}$ from $\Omega$ to ${\rm Diff}^2(M)$:
\begin{align*}
\mathscr{F}:~~\Omega &  ~~\longrightarrow ~~{\rm Diff}^2(M)\\ ~~\omega &~~ \longmapsto ~~f_{\omega}
\end{align*}
with property $\mathscr{F}(\omega_f)=f$ for some $\omega_f \in \Omega$.

%Given some probability $\nu$ of $\Omega$ supported on
%some small neighborhood of $\omega_f$. Then the forward random orbits associated to $\nu$ are the sequences
%$f_{\omega_{n-1}}\circ f_{\omega_{n-2}}\cdots \circ f_{\omega_0}(x)$,
%where $\omega_{n-1}$, $\omega_{n-2}, \cdots, \omega_0$ are independent random variables with distribution law $\nu$.
%The backward random orbits are defined similarly.
%We are interested in study the dynamical systems generated by both forward and backward random orbits.
For each $x\in M$, define $\ell_x:\Omega \to M$ such that $\ell_x(\omega)=\cF(\omega)(x)$ for every $\omega \in \Omega$.
Let us introduce a sequence of Borel probability measures $\{\nu_{\varepsilon}\}_{\varepsilon>0}$ on $\Omega$ satisfying:
\begin{itemize}
	\item[\bf{(H1}):] $\{\rm supp(\nu_{\varepsilon})\}_{\varepsilon>0}$ form a nested family of compact connected sets such that
    ${\rm supp}(\nu_{\varepsilon})\rightarrow \{\omega_{f}\},$
	when $\varepsilon\rightarrow 0$.
    \item[\bf{(H2}):]  For each $\varepsilon>0$, $(\ell_x)_{\ast}\nu_{\varepsilon}\ll {\rm Leb}$ for each $x\in M$, where $(\ell_x)_{\ast}\nu_{\varepsilon}(A)=\nu(\ell_x^{-1}A)$ for any measurable subset $A$.
\end{itemize}

%For every $\varepsilon>0$, we write $\mathcal{N}_{\varepsilon}={\rm supp}(\nu_{\varepsilon})$.
%, also denote by $P_{\varepsilon}:=\nu_{\varepsilon}^{\ZZ}$ the Borel product probability measure of the product space $\Omega^{\mathbb{Z}}$, thus ${\rm supp}(P_{\varepsilon})=\mathcal{N}_{\varepsilon}^{\mathbb{Z}}$.

\begin{Definition}
We refer to $\mathscr{R}_f=\{\mathscr{F}, (\nu_{\varepsilon})_{\varepsilon>0}\}$ with $\mathscr{F}$ having above settings and $\{\nu_{\varepsilon}\}_{\varepsilon>0}$ satisfying $\bf{(H1)}$ and $\bf{(H2)}$
as a regular random perturbation of $f$. Given $\varepsilon>0$, $\mathscr{R}_f^{\varepsilon}=\{\mathscr{F},\nu_{\varepsilon}\}$ is denoted to be the $\varepsilon$-regular random perturbation of $f$ with $\varepsilon$ as a noise level.
\end{Definition}

\begin{Remark}
We will also use $\mathscr{R}_f=\{\mathscr{F}, (\nu_n)_{n\in \NN}\}$ to denote a regular random perturbation of $f$ if $\{\nu_n\}_{n\in \NN}$ satisfies $\bf{(H1)}$ and $\bf{(H2)}$ in discrete sense: ${\rm supp}(\nu_n)\to \{\omega_f\} $ as $n\to +\infty$, and $(\ell_x)_{\ast} \nu_n\ll {\rm Leb}$ for every $x\in M$ and $n\in \NN$.
\end{Remark}

Let us mention that the regular random perturbations of $C^2$ diffeomorphisms are always exist, see \cite[Example 5.2]{Alv03} for details.
\bigskip

For each $\underline{\omega}=(\cdots, \omega_{-1},\omega_{0},\omega_{1}\cdots)\in \Omega^{\mathbb{Z}}$,
we use the presentation
$$
f_{\underline{\omega}}^n=\left\{
\begin{array}{cl}
  f_{\omega_{n-1}}\circ\cdots \circ f_{\omega_{0}} & ~~n>0~; \\
  id &~~ n=0~;\\
  f_{\omega_n}^{-1}\circ\cdots \circ f_{\omega_{-1}}^{-1} &~~n<0~.
\end{array}\right.
$$
As an important idea to investigate the dynamics under random perturbations, it is useful to introduce the two-sided skew-product map
\begin{align*}
G: \Omega^{\mathbb{Z}}\times M & ~~\longrightarrow ~~\Omega^{\mathbb{Z}}\times M \\
(\underline{\omega}~,~x) &~~\longmapsto ~~(\sigma(\underline{\omega}),f_{\omega_0}(z)).
\end{align*}
where $\sigma$ is the left shift operator.
Let us define the projection maps
\begin{align*}
\PP: \Omega^{\mathbb{Z}}\times M & ~~\longrightarrow ~~\Omega^{\mathbb{N}}\times M \\
(\underline{\omega}~,~x) &~~\longmapsto ~~(\underline{\omega}^{+}~, ~x)
\end{align*}
\begin{align*}
%\textbf{P}_1: \Omega^{\mathbb{Z}}\times M & ~~\longrightarrow ~~\Omega^{\mathbb{Z}} \\
%(\underline{\omega}~,~x) &~~\longmapsto ~~\underline{\omega}
%\end{align*}
%\begin{align*}
\PP_M: \Omega^{\mathbb{Z}}\times M & ~~\longrightarrow ~~M \\
(\underline{\omega}~,~x) &~~\longmapsto ~~x
\end{align*}
where $\underline{\omega}^{+}=(\omega_0,\omega_1,\cdots)\in \Omega^{\mathbb{N}}$ for $\underline{\omega}=(\cdots,\omega_{-1},\omega_0,\omega_1,\cdots)\in \Omega^{\mathbb{Z}}$.

\begin{Definition}
Given a Borel probability $\nu$ on $\Omega$, we say $\mu$ is a stationary measure for $\nu$ if
%Given $\varepsilon>0$, a Borel probability measure $\mu_{\varepsilon}$ is said to be \emph{invariant} or \emph{stationary} associated to $\mathscr{R}_f^{\varepsilon}$, if
$$
\mu(A)=\int \mu(f_{\omega}^{-1}A)d\nu(\omega)
$$
for every Borel subset $A$.
\end{Definition}
%We will also say that $\mu$ is a stationary measure of $\mathscr{R}_f^{\varepsilon}=\{\mathscr{F},\nu_{\varepsilon}\}$ or $\mathscr{R}_f^n=\{\mathscr{F},\nu_n\}$ if it is a stationary measure for $\nu_{\varepsilon}$ or $\nu_n$, respectively.

For given regular random perturbation $\mathscr{R}_f=\{\mathscr{F}, (\nu_{\varepsilon})_{\varepsilon>0}\}$, one says that $\mu$ is a stationary measure for $\mathscr{R}_f$ if it is a stationary measure for some $\nu_{\varepsilon}$. Furthermore, we say $\mu$ is a stationary measure for $\mathscr{R}_f^{\varepsilon}=\{\mathscr{F}, \nu_{\varepsilon}\}$ as long as it is a stationary measure for $\nu_{\varepsilon}$.
Note that when $\nu$ is reduced to $\delta_{\omega_f}$, then the stationary measures for $\nu$ are $f$-invariant measures. Moreover, we have the following result, see \cite[Lemma 5.5]{Alv03} for a proof.

\begin{Lemma}
Let $\mathscr{R}_f=\{\mathscr{F},(\nu_{\varepsilon})_{\varepsilon>0}\}$ be a random perturbation of $f$. If $\mu_{\varepsilon}$ is a stationary measure for $\nu_{\varepsilon}$ for any $\varepsilon>0$, then all the accumulation points of $\{\mu_{\varepsilon}\}_{\varepsilon>0}$ as $\varepsilon\to 0$ are $f$-invariant measures.
\end{Lemma}

%When associated with a stationary measure $\mu_{\varepsilon}$, $\mathscr{R}_f^{\varepsilon}$ will be referred to as $\mathscr{R}_f^{\varepsilon} (\mu_{\varepsilon})$.

%Let $\mu_1$, $\mu_2$ are two measures on $M$, we write $``\mu_1 \ll \mu_2 "$ to denote that $\mu_1$ is absolutely continuous with respect to $\mu_2$.

We need the following observation:
\begin{Lemma}\label{stationaryabs}
Let $\mathscr{R}_f$ be a regular random perturbation of $f$, then any the stationary measure for $\mathscr{R}_f$ is absolutely continuous with respect to Lebesgue measure.
\end{Lemma}

This follows from the hypothesis $\bf{(H2)}$ of regular random perturbation. Indeed, for any probability $\nu$ on $\Omega$ and stationary measure $\mu$ for $\nu$, the definition of stationary measure yields that if $(\ell_x)_{\ast}\nu\ll {\rm Leb}$ for every $x\in M$, then $\mu\ll {\rm Leb}$.
\smallskip

For a stationary measure $\mu$ for some Borel probability $\nu$ on $\Omega$, we say a Borel set $A$ is \emph{random invariant} if for $\mu$-a.e. $x\in M$,
$$
x\in A \Longrightarrow f_{\omega}(x)\in A \quad \textrm{for} ~\nu\text{-a.e.}~\omega,
$$
$$
x\in M\setminus A \Longrightarrow f_{\omega}(x)\in M\setminus A \quad \textrm{for} ~\nu\text{-a.e.}~\omega.
$$

\begin{Definition}
A stationary measure $\mu$ is ergodic if every random invariant set has either $\mu$ measure $0$ or $\mu$ measure $1$.
\end{Definition}

By the ergodic decomposition theorem in random case (see e.g. \cite[Appendix A.1]{kiferergodic},\cite[Proposition 5.9]{Alv03}), every stationary measure can be written as a convex combination of ergodic stationary measures. Therefore, the existence of ergodic stationary measures is guaranteed. We introduce the following notion, as a special kind of zero noise limits.

\begin{Definition}\label{randomlylimit}
Given $f\in {\rm Diff}^2(M)$ and an $f$-invariant measure $\mu$, if there is a regular
random perturbation $\mathscr{R}_f=\{\mathscr{F}, (\nu_{\varepsilon})_{\varepsilon>0}\}$ of $f$ such that
$\lim_{n\to +\infty}\mu_{\varepsilon_n}=\mu$ for some subsequence $\{\varepsilon_n\}$, where $\mu_{\varepsilon_n}$ is a stationary measure for $\nu_{\varepsilon_n}$ for every $n\in \NN$.
Then $\mu$ is said to be a randomly ergodic limit.
\end{Definition}

\begin{Lemma}\label{Lem:invariant-measure-G}
For any stationary measure $\mu$ for $\nu$, there is a unique $G$-invariant measure $\mu^G$ on $\Omega^\ZZ\times M$ such that $\PP_*\mu^G=\nu^\NN\times \mu$. We have the following properties:
\begin{itemize}
\item $\mu$ is an ergodic stationary measure if and only if $\mu^G$ is ergodic for $G$.

\item Suppose that $\mu_n$ is a stationary measure for $\nu_n$ for any $n\in \NN$, if $\mu_n\to\mu$, $\nu_n\to\nu$ as $n\to +\infty$, then $\mu_n^G\to \mu^G$ as $n\to +\infty$.
\end{itemize}
\end{Lemma}

\begin{proof} It follows form \cite[Proposition 1.2]{LiQ95} that for every stationary measure $\mu$ for some probability $\nu$ on $\Omega$, there exists the unique $G$-invariant measure $\mu^G$ such that $\PP_*\mu^G=\nu^\NN\times \mu$. Moreover, by \cite[Proposition 1.3]{LiQ95}
we have that $\mu$ is ergodic if and only if $\mu^G$ is ergodic for $G$.

Under the assumptions of the second item,
by passing to the limit, one gets that $\mu$ is a stationary measure for $\nu$.
%Indeed, for each continuous function $\varphi$ on $M$, by assumption for each $n\in \mathbb{N}$ we have
%$$\int \varphi(x)d\mu_n(x)=\int \int \varphi(f_{\omega}x)d\nu_n(\omega)d\mu_n(x).$$
%By passing to the limit, one gets that
%$$
%\int \varphi(x)d\mu(x)=\int \int \varphi(f_{\omega}x)d\nu(\omega)d\mu(x).
%$$
%Thus, $\mu$ is a stationary measure for $\nu$.
Assume that $\lim_{k\to \infty}\mu_{n_k}^G=\eta$ for some subsequence $\{n_k\}$. Since each $\mu_n^G$ is $G$-invariant, $\eta$ is $G$-invariant.
Moreover, by applying the continuity of $\PP$ one gets
$$
\PP_{\ast}\eta=\PP_{\ast}\lim_{k\to \infty}\mu_{n_k}^G=\lim_{k\to \infty}\PP_{\ast}\mu_{n_k}^G=\lim_{k\to \infty} \nu_{n_k}^{\mathbb{N}}\times \mu_{n_k}=\nu^{\NN}\times \mu.
$$
Since $\mu^G$ is the only $G$-invariant measure such that $\PP_{\ast}\mu^G=\nu^{\NN}\times \mu$, we obtain $\eta=\mu^G$ and thus $\lim_{n\to \infty}\mu_n^G=\mu^G$.

\end{proof}

As a consequence of Lemma \ref{Lem:invariant-measure-G}, we have:

\begin{Corollary}\label{toinv}
Let $\mu_n$ be a sequence of stationary measures for $\nu_n$, if $\nu_n\to \delta_{\omega_f}$ and $\mu_n \to \mu$ as $n\to +\infty$, then $\lim_{n\to \infty}\mu^G_n=\delta_{\omega_f}^{\mathbb{Z}}\times \mu$. In particular, if $\mu$ is a randomly ergodic limit, then $\mu^G=\delta_{\omega_f}^{\mathbb{Z}}\times \mu$.
\end{Corollary}

\subsection{Dominated splittings under random perturbations}

Let $\Lambda$ be an attractor for a $C^2$ diffeomorphism $f$ with trapping region $U$, and there exists the dominated splitting $T_{\Lambda}M=E\oplus F$. Fixing a regular random perturbation $\mathscr{R}_f=\{\mathscr{F}, (\nu_{\varepsilon})_{\varepsilon>0}\}$.
Since $f(\overline{U})\subset U$, we can choose $\varepsilon_0$ small enough such that: for every $\varepsilon\le \varepsilon_0$, $f_{\underline{\omega}}(\overline{U})\subset U$ for any $\underline{\omega}\in \mathcal{N}_{\varepsilon_0}^{\mathbb{Z}}$ as well.
For every $\underline{\omega}=(\cdots,\omega_{-1},\omega_{0},\omega_{1}\cdots)\in\mathcal{N}_{\varepsilon}^{\mathbb{Z}}$, $\varepsilon\le \varepsilon_0$, we write
$\Lambda_{\underline{\omega}}=\bigcap_{n\ge1}f_{\sigma^{-n}\underline{\omega}}^n(\overline{U})$.
By construction, $\{\Lambda_{\underline{\omega}}\}$ is a family of compact subsets satisfying $f_{\omega_0}\Lambda_{\underline{\omega}}=\Lambda_{\sigma\underline{\omega}}$ for each $\underline{\omega}\in \mathcal{N}_{\varepsilon}^{\mathbb{Z}}$.
For each $\varepsilon\le \varepsilon_0$, we define
$
\Lambda_{\varepsilon}=\bigcup_{\underline{\omega}\in \mathcal{N}_{\varepsilon}^{\mathbb{Z}}}\{\underline{\omega}\}\times \Lambda_{\underline{\omega}},
$
it is a $G$-invariant compact subset contained in $\mathcal{N}_{\varepsilon}^{\mathbb{Z}}\times U$.
By the choice of $\varepsilon_0$, for every $\varepsilon\le \varepsilon_0$,
any stationary measure for $\nu_{\varepsilon}$ supported on $U$ is concentrated on $\PP_M(\Lambda_{\varepsilon})$ and the corresponding G-invariant measure $\mu^G$ is concentrated on $\Lambda_{\varepsilon}$.

%Since the distributions $E$ and $F$ are continuous, we may extend them continuously
%to some neighborhood $U_0\subset U$ of $\Lambda$, denoted by $E$ and $F$ as well. Up to shrinking $\varepsilon_0$ if necessary we will assume that for any $\varepsilon \le \varepsilon_0$,
%$$
%\Lambda_{\varepsilon}\subset \mathcal{N}_{\varepsilon}^{\mathbb{Z}}\times U_0.
%$$

By standard cones argument as in \cite[\S 4.3]{CY05} one has
\begin{Lemma}\label{randomconti} For each $\varepsilon\le \varepsilon_0$,  there are random sub-bundles $E(\underline{\omega},x)$ and $F(\underline{\omega},x)$ for every $(\underline{\omega},x)\in \Lambda_{\varepsilon}$ corresponding to $E$ and $F$. Moreover,
$E(\underline{\omega},x)$ and $F(\underline{\omega},x)$ depend continuously on $(\underline{\omega},x)$.
\end{Lemma}

%Furthermore, by using $C^r$-section theorem \cite[Theorem 6.1]{HP68}, with a slightly modification of \cite[Proposition 6.2]{ADLP17} we get the following property (see Appendix \ref{se7} for a detailed proof ).
%
%\begin{Lemma}\label{h}
%Then there exists $H>0$ and $0<\eta<1$, such that for any $\varepsilon \le \varepsilon_0$ we have
%$$
%d(E(\underline{\omega},x), E(\underline{\omega},y))\le H d(x,y)^{\eta}
%$$
%whenever $(\underline{\omega},x), (\underline{\omega},y)\in \Lambda_{\varepsilon}$.  The same result holds true for random invariant sub-bundle $F$.
%\end{Lemma}

%\begin{Definition}
%We define the $E$-direction cone field
%$K_a^E=\{K_a^E(x)\}$ of width $a$ over $U_0$ by
%$$
%K_a^E(x)=\left\{\nu=\nu_E+\nu_F \in E(x)\oplus F(x), \frac{\|v_F\|}{\|v_E\|}\le a\right\}
%$$
%for every $x\in U_0$. One can define the $F$-direction cone field by reversing the roles of $E$ and $F$.
%\end{Definition}
%
%
%Let us say that a smooth embedded sub-manifold $D$ is \emph{tangent to} $K_a^E$ if it has dimension coincide with ${\rm dim}E$ and $T_xD\subset K_a^E(x)$ for any $x\in D$.
Recall the notion of local unstable disk in deterministic dynamical systems.
We represent $W^{E,u}_{\delta}(x)$ as the \emph{local unstable disk} of $x$ with radius $\delta$ around $x$, which is tangent to $E$ everywhere and $(C,\lambda)$-\emph{backward contracting} for some $C>0$ and $\lambda\in (0,1)$, i.e.,
$$
d(f^{-n}(y), f^{-n}(z))\le C\lambda^n d(y,z) \quad \textrm{for every}~n\in \NN,
$$
whenever $y,z \in W_{\delta}^{E,u}(x)$.

Similarly, using $W^{E,u}_{\delta}(\underline{\omega},x)$ to denote the \emph{random local unstable disk} of $(\underline{\omega},x)$ with radius $\delta$ around $x$, which is tangent to the random sub-bundle $E(\underline{\omega},\cdot)$ everywhere and $(C,\lambda)$-\emph{backward contracting} for some $C>0, \lambda\in (0,1)$, in the sense that
$$
d(f_{\underline{\omega}}^{-n}(y), f_{\underline{\omega}}^{-n}(z))\le C\lambda^n d(y,z) \quad \textrm{for every}~n\in \NN,
$$
whenever $y,z \in W_{\delta}^{E,u}(\underline{\omega},x)$. Note that we have $W_{\delta}^{E,u}(\underline{\omega},x)\subset \Lambda_{\underline{\omega}}$ if $(\underline{\omega},x)\in \Lambda_{\varepsilon_0}$.
%Similarly, we represent $W^{E,u}_{\delta}(x)$ the local unstable disk of $x$ with radius $\delta$ around $x$, which is tangent to $E$ and $(C,\lambda)$-\emph{backward contracting} (for some $C>0$ and $\lambda\in (0,1)$), i.e.,
%$$
%d(f^{-n}(y), f^{-n}(z))\le C\lambda^n d(y,z) \quad \textrm{for every}~n\in \NN,
%$$
%whenever $y,z \in W_{\delta}^{E,u}(x)$.

Sometimes, one needs to look random unstable disks on the product space $\Omega^{\ZZ}\times M$. For this,
we use $\{\underline{\omega}\}\times W^{E,u}_{\delta}(\underline{\omega},x)\subset \{\underline{\omega}\}\times M$ to denote the endowed random local unstable disk associated to $W^{E,u}_{\delta}(\underline{\omega},x)$.

\smallskip
Now we state a result which asserts the uniform Lipschitz continuity of random local unstable disks, it is a simple generalization of its deterministic version (see e.g. \cite[$\S$ 6.1]{HP68}). (One can see a direct proof in Appendix \ref{se7}).
\begin{Lemma}\label{hee}
Given $C>0$ and $0<\lambda<1$ there exists constant $L_0>0$, such that for any $(C,\lambda)$-backward contracting random local unstable disk $\gamma=W^{E,u}_{\delta}(\underline{\omega},x)$ of $(\underline{\omega},x)\in \Lambda_{\varepsilon_0}$, we have
$$
\rho(T_y\gamma,T_z\gamma)\le L_0d(x,y)
$$
for any $y,z\in \gamma$.
\end{Lemma}

We shall consider following typical (random) local unstable disks.
\begin{Definition}\label{bck}
Given $\delta>0$, $\lambda\in (0,1)$,
we say a random local unstable disk $W^{E,u}_{\delta}(\underline{\omega},x)$ of $(\underline{\omega},x)$ is $\lambda$-backward contracting if
\begin{itemize}
\item $W^{E,u}_{\delta}(\underline{\omega},x)$ is $(1, \lambda)$-backward contracting;
\item $\|Df^{-n}|_{E(\underline{\omega},y)}\| \le \lambda^n$ for every $n\in \NN$ and $y\in W^{E,u}_{\delta}(\underline{\omega},x)$.
\end{itemize}
\end{Definition}

%From the definition, we know that $x$ is backward contracting if and only if $(\underline{\omega^f},x)$ is backward contracting.
%%It is well known that if $E$ is uniformly expandi744ng, and consider the random perturbations, there exists random unstable manifolds $W^{E,u}$ tangent to the random invariant bundle $E(\underline{\omega},x)$ everywhere.
%\begin{Definition}
%Given $0<\lambda<1$, a point $(\underline{\omega},x)$ is said to be \emph{randomly $\lambda$-backward contracting} if
%$$
%\prod_{i=0}^{n-1} \|Df^{-1}_{\sigma^{-i}\underline{\omega}}|_{E(G^{-i}(\underline{\omega},x))}\|\le \lambda^{n}\quad \text{for any}~ n\in\NN.
%$$
%Similarly, we say $x\in M$ is $\lambda$-backward contracting if
%$$
%\prod_{i=0}^{n-1}\|Df^{-1}|_{E(f^{-i}x)}\| \le \lambda^n \quad \textrm{for any}~n\in \NN.
%$$
%\end{Definition}
%
%For convenience, we define the sets of backward contracting points as follows:
%$$
%\mathscr{H}^G_{\lambda}=\left\{(\underline{\omega}, x) : (\underline{\omega},x)~\textrm{is}~\lambda-\textrm{backward contracting}\right\}
%$$
%
%$$
%\mathscr{H}_{\lambda}=\left\{x :  \prod_{i=0}^{n-1}\|Df^{-1}|_{E(f^{-i}x)}\| \le \lambda^n, \quad \forall n\in \mathbb{N}\right\}
%$$
Similarly, a local unstable disk $W^{E,u}_{\delta}(x)$ is said to be \emph{$\lambda$-backward contracting} whenever
it is $(1, \lambda)$-backward contracting and $\|Df^{-n}|_{E(y)}\| \le \lambda^n$ for every $n\in \NN$ and $y\in W^{E,u}_{\delta}(x)$.

\subsection{Random Gibbs $cu$-states}
Throughout this subsection, let $f\in {\rm Diff}^2(M)$ and fix a stationary measure $\mu$ of $\mathscr{R}_f^{\varepsilon}=\{\cF, \nu_{\varepsilon}\}$ for some small $\varepsilon>0$ and assume $\mu^G$ is the $G$-invariant measure associated to $\mu$ given by Lemma \ref{Lem:invariant-measure-G}.

According to the hypotheses \textbf{(H1)}, $\mathcal{N}_{\varepsilon}={\rm supp}(\nu_{\varepsilon})$ is compact, then use the continuity of $\mathscr{F}$, we have that $\mathscr{F}(\mathcal{N}_{\varepsilon})$ is a compact subset of ${\rm Diff}^2(M)$ around $f$, which implies that
\begin{equation}\label{forLY}
\int \left(\log^{+}|f_{\omega}|_{C^1}+\log^{-}|f_{\omega}^{-1}|_{C^1}\right){\rm d}\nu_{\varepsilon}(\omega)<+\infty,
\end{equation}
\begin{equation}\label{forMA}
\int \log^{+}|f_{\omega}^{-1}|_{C^2}{\rm d}\nu_{\varepsilon}(\omega)<+\infty,
\end{equation}
where we use the notation $\log ^{+}a=\max\{\log a,0\}$ and $\log^{-}a=\max\{-\log a,0\}$.
%\marginpar{(3) is the condition of random Oseledec's theorem, (4) is the condition of the random Pesin unstable manifold theroem}
%where we use the notation $\log ^{+}a=\max\{\log a,0\}$ and $\log^{-}a=\max\{-\log a,0\}$.

%Throughout this subsection, let us fix a stationary measure $\mu$ of $\mathscr{R}_f^{\varepsilon}$ for some small $\varepsilon>0$. Let $\mu^G$ be the  $G$-invariant measure associated to $\mu$ as state in Lemma \ref{Lem:invariant-measure-G}.

In terms of $(\ref{forLY})$, we have the following proposition \cite[Chapter VI: Proposition 1.2]{LiQ95}
\begin{Proposition}
For $\mu^G$-a.e. $(\underline{\omega},x)$, there exists $r(x)\in \mathbb{N}$ and a decomposition
$$
  T_xM=\bigoplus_{i=1}^{r(x)}E_i(\underline{\omega},x)
 $$
such that there are numbers $\lambda_1(x)>\cdots.\lambda_{r(x)}(x)$ satisfying
$$
\lim_{n\rightarrow \pm \infty}\frac{1}{n}\log \|Df_{\underline{\omega}}^{n}(x)v\|=\lambda_{i}(\underline{\omega},x)=\lambda_{i}(x)
$$
for any non-zero vector $v\in E_i(\underline{\omega},x)$, and every $1\le i \le r(x)$.
\end{Proposition}

\begin{Definition}
The numbers $\lambda_i(x)$, $1\le i \le r(x)$ we introduced above are called the random Lyapunov exponents of $\mathscr{R}_f^{\varepsilon}$.
\end{Definition}

Since condition $(\ref{forMA})$ holds, corresponding to the positive(resp. negative) random Lyapunov exponents, one may construct random unstable(resp. stable) manifolds analogous in deterministic dynamical systems. See \cite[Chapter VI: Proposition 1.4]{LiQ95} for detalis.

\begin{Proposition}\label{unstablerandom}
For $\mu^{G}$-a.e. $(\underline{\omega},x)$, if
$\lambda_1(x)>\cdots >\lambda_d(x)$ are the positive random Lyapunov exponents of $(\underline{\omega},x)$, then there exist random unstable manifolds $W^{u,i}(\underline{\omega},x)$,$1\le i \le d$ defined by
$$
W^{u,i}(\underline{\omega},x)=\left\{y\in M: \limsup_{n\rightarrow +\infty}\frac{1}{n}\log d(f_{\underline{\omega}}^{-n}(x),f_{\underline{\omega}}^{-n}(y))\le -\lambda_i(x)\right\}.
$$
In addition, for every $1\le i \le d$, $W^{u,i}(\underline{\omega},x)$ is a $C^1$ sub-manifold tangent to $\bigoplus_{j=1}^{i} E(\underline{\omega},x)$ at $(\underline{\omega},x)$
with dimension equal to {\rm dim}$(\oplus_{j=1}^{i} E(\underline{\omega},x))$. Also, we have
$W^{u,1}(\underline{\omega},x)\subset \cdots \subset W^{u,d}(\underline{\omega},x)$.
\end{Proposition}

Now we assume that $\Lambda$ is an attractor of $f$ with dominated splitting $T_{\Lambda}M=E\oplus_{\succ}F$. Shrinking $\varepsilon$ if necessary, we assume $\varepsilon$ is less than $\varepsilon_0$ given in Lemma \ref{randomconti},
it follows that for $\mu^G$-a.e. $(\underline{\omega},x)$ there exists an integer $1\le d(x) \le r(x)-1$ such that
$$
E(\underline{\omega},x)=\bigoplus_{i=1}^{d(x)}E_i(\underline{\omega},x),~~ F(\underline{\omega},x)=\bigoplus_{i=d(x)+1}^{r(x)}E_i(\underline{\omega},x).
$$
Moreover, if the Lyapunov exponents of $\mu^G$ along $E$ are positive, namely $\lambda_{d(x)}(x)>0$ for $\mu^G$-a.e. $(\underline{\omega},x)$, then there exists a Pesin unstable manifold $W^{E,u}(\underline{\omega},x):=W^{u,d(x)}(\underline{\omega},x)$ tangent to $E$ for $\mu^G$-almost every $(\underline{\omega},x)$.

Analogous to the deterministic setting, we can introduce the measurable partition subordinate to Pesin unstable manifolds (in different levels), follows after \cite[Chapter VI: E]{LiQ95}.
\begin{Definition}
$\xi$ is called a measurable partition \emph{subordinate} to $W^{u,i}$ w.r.t. $\mu^G$ if for $\mu^G$-a.e $(\underline{\omega},x)$, $\xi(\underline{\omega},x)\subset \{\underline{\omega}\}\times M$, and $\xi_{\underline{\omega}}(x)=\{y:(\underline{\omega},y)\in \xi(\underline{\omega},x)\}\subset W^{u,i}(\underline{\omega},x)$ contains a neighborhood of $x$ open in $W^{u,i}(\underline{\omega},x)$.
\end{Definition}

Similarly, one can define the measurable partitions subordinate to $W^{E,u}$. Now we can give the precise definition of random Gibbs $cu$-state as below.

\begin{Definition}
We say $\mu^G$ is a Gibbs cu-state (associated to $E$), if the Lyapunov exponents of $\mu^G$ along $E$ are positive and  $\mu^G$ has absolutely continuous conditional measures along $\xi$ for any
measurable partition $\xi$ subordinate to $W^{E,u}$.
\end{Definition}

As a result of the regularity of the random perturbations, we have the following important fact.

\begin{Lemma}\label{rGibbs}
If the Lyapunov exponents of $\mu^G$ along $E$ are positive, then
$\mu^G$ is the random Gibbs cu-states associated to $E$.
\end{Lemma}
\begin{proof}
We sketch the proof of this lemma.
By Lemma \ref{stationaryabs}, $\mu$ is absolutely continuous w.r.t. Lebesuge measure. This together with Theorem 1.1 of \cite[Chapter VI]{LiQ95} imply that $\mu$ has absolutely continuous conditional measures along any measurable partition subordinate to $W^u$. Use the same manner of \cite[Lemma 9.11]{LyA2} one can construct two measurable partitions $\xi \prec \xi^E$ ($\xi^E$ is finer than $\xi$) such that $\xi$ and $\xi^E$ are subordinate to $W^{u}$ and $W^{E,u}$, respectively. Thus, the essential uniqueness of conditional measures \cite[Exercise 5.21]{VO} implies  that $\mu$ has absolutely continuous conditional measures along $\xi^E$,  this shows that $\mu^G$ is a random Gibbs $cu$-state associated to $E$ by definition.
\end{proof}

Now we discuss the densities of conditional measures of random Gibbs $cu$-states along measurable partitions consisted of random local unstable disks that are backward contracted by a uniform rate.

\begin{Lemma}\label{fde}
Let $\cP$ be a measurable partition formed by disjoint endowed random local unstable disks with uniform radius, and all these random local unstable disks are $(C,\sigma)$-backward contracting for some $C>0$ and $\sigma\in (0,1)$.
Then there exists $K>0$ such that for any random Gibbs cu-state admitting positive measure on the union of the elements from $\cP$, the densities of the conditional measures of $\mu$ along $\cP$ are between $1/K$ and $K$.
\end{Lemma}

\begin{proof}
Rewrite
$$
\cP=\big\{\{\underline{\omega}\}\times \gamma: ~\textrm{each}~ \gamma~ \textrm{is}~(C,\sigma)\textrm{-backward contracting}\big\}.
$$
Let $P$ be the union of the element of $\cP$.  
Let $\alpha, \beta$ be the constants satisfying that for each $\{\underline{\omega}\}\times \gamma \in \cP$,
\begin{itemize}
\item $d(y,z)\le \alpha$ for every $y,z\in \gamma$;
\item $1/\beta \le {\rm Leb}(\gamma)\le \beta$.
\end{itemize}
%Use $A$ as a upper bound of the radius for random local unstable disks of $\cP$, and let $B$ be the constant for which the Lebesgue measure of any random local unstable disks is bounded below from $1/B$ and away from $B$.
Since $f_{\underline{\omega}}$ are $C^2$ diffeomrophisms nearby $f$ with a fixed radius, thus there exists a uniform bound on the Lipschitz constant of $Df_{\underline{\omega}}$. Combining this fact with Lemma \ref{hee} yields that for every $y,z$,
$$
|\log J^E(\underline{\omega},x)-\log J^E(\underline{\omega},y)|\le L_1d(x,y)
$$
for some constant $L_1>0$ depending only on $f$.
Therefore, for any random local unstable disk $\gamma$ which is $(C,\sigma)$-backward contracting, for every $y,z\in \gamma$ and any $n\in \NN$,
\begin{eqnarray*}
\sum_{j=1}^n\left|\log \frac{J^E(G^{-j}(\underline{\omega},y))}{J^E(G^{-j}(\underline{\omega},z))}\right|&=& \sum_{j=1}^{n}\left|\log J^E(G^{-j}(\underline{\omega},y))-J^E(G^{-j}(\underline{\omega},z))\right|\\
&\le &L_1\sum_{j=1}^n d(f_{\underline{\omega}}^{-j}(y), f_{\underline{\omega}}^{-j}(z))\\
&\le &L_1\sum_{j=1}^n C\sigma^j d(y,z)\\
&\le & \frac{L_1C\sigma \alpha}{1-\sigma}:=\log M.
\end{eqnarray*}
Consequently, we obtain
\begin{equation}\label{fd}
\frac{1}{M}\le \prod_{j=1}^{\infty}\frac{J^E(G^{-j}(\underline{\omega},y))}{J^E(G^{-j}(\underline{\omega},z))}\le M.
\end{equation}
Now let us take $\mu$ as a random Gibbs cu-state such that $\mu(P)>0$.
Let $\rho$ be the density of the conditional measures of $\mu$ along $\cP$ w.r.t. its Lebesgue measures.
Moreover, for $\hat{\mu}$-almost every $\{\underline{\omega}\}\times \gamma \in \cP$, the density $\rho$ has following formula (see \cite[ Chapter VI: Proposition 2.2
and Corollary 8.1]{LiQ95} or \cite[Proposition 2.1]{cv}):
$$
\frac{\rho(\underline{\omega},y)}{\rho(\underline{\omega},z)}=\prod_{j=1}^{\infty}\frac{J^E(G^{-j}(\underline{\omega},z))}
{J^E(G^{-j}(\underline{\omega},y))}
$$
for $\mu_{\{\underline{\omega}\}\times \gamma}$-almost every $(\underline{\omega},y)$ and $(\underline{\omega},z)$ in $\{\underline{\omega}\}\times \gamma$.
Moreover, we have
\begin{equation}\label{rho}
\rho(\underline{\omega},z)=\frac{\kappa((\underline{\omega},y),(\underline{\omega},z))}{\displaystyle{\int_\gamma \kappa((\underline{\omega},y).(\underline{\omega},z))d{\rm Leb}_{\gamma}(z)}},
\end{equation}
where
$$
\kappa((\underline{\omega},y),(\underline{\omega},z))=\prod_{j=1}^{\infty}\frac{J^u(G^{-j}(\underline{\omega},y))}{J^u(G^{-j}(\underline{\omega},z))}.
$$
By the choices of $\beta$ and the estimate (\ref{fd}),  the formula (\ref{rho}) yields
$$
\frac{1}{\beta M^2}\le \rho(\underline{\omega},z)\le \beta M^2.
$$
Thus, it suffices to take  $L=\beta M^2$ to complete the proof.

\end{proof}

\section{Absolutely continuous conditional measures}\label{se4}
%The main goal of section is to give the detailed proof of Theorem \ref{bac}.
As the main technical step for proving Theorem \ref{m2} and Theorem \ref{main}, we shall make effort to show the following result in this section.
\begin{theoremalph}\label{bac}
Let $f$ be a $C^2$ diffeomorphism with an attractor $\Lambda$ admitting a dominated splitting $T_{\Lambda} M=E\oplus F$.

If $\mu$ is a randomly ergodic limit satisfying $\lambda_E(\mu)>0$, then there exists $\delta>0$, $0<\sigma<1$ and a measurable partition $\cF$ with following properties:
\begin{itemize}
\item the union of elements from $\cF$ has positive $\mu$-measure;
\item each $\gamma\in \cF$ is a $\sigma$-backward contracted local unstable disk with radius $\delta$;
\item $\mu$ has absolutely continuous conditional measures along $\cF$.
\end{itemize}
\end{theoremalph}

We will complete the proof of Theorem \ref{bac} in $\S$\ref{se44}, here
we briefly outline the strategy for proving Theorem \ref{bac}. To begin with, let us fix ergodic stationary measures $\mu_n$ converges to $\mu$, consider $\mu_n^G$ and $\mu^G$ as the corresponding $G$-invariant measures of $\mu_n$ and $\mu$ introduced in Lemma \ref{Lem:invariant-measure-G}.

Firstly, we show that there exists a ``Pesin block" $\cH^G_{\lambda_0}$ for some $\lambda_0\in (0,1)$ in $\Omega^{\ZZ}\times M$ possessing positive $\mu_n^G$ measure for all sufficiently large $n$, which consists of points admitting random local unstable disks tangent to $E$ of uniform size and backward contraction rate. This is done in $\S$\ref{se41} and $\S$\ref{se43}.

Secondly, we get a measurable partition $\cF^G$ formed by random local unstable disks generated by points from $\cH^G_{\lambda_0}$, which is induced by a foliated chart established in lift space $\Omega^{\ZZ}\times M$, this is a consequence of Proposition \ref{forc}.

Finally, by the construction of $\cF^G$ one knows that the densities of disintegrations of $\mu_n^G$ along $\cF^G$ are uniformly bounded from above and from below by applying Lemma \ref{fde}, then passing to the limit as noise level tends zero, we get the desired absolute continuity of $\mu$ along a measurable partition $\cF$, where $\cF$ is a family of local unstable disks arising from $\cF^G$ .

%We introduce the ``Pesin block" $\cH^G_{\lambda_0}$ for some $\lambda_0\in (0,1)$ in $\Omega^{\ZZ}\times M$, which consists of points admitting random local unstable disks tangent to $E$ with uniform radius and backward contracting properties, this is discussed in, we show that $\cH^G_{\lambda_0}$ has not only positive $\mu^G$ measure but also positive $\mu_n^G$ measure for sufficiently large $n$ (Theorem \ref{main2}). Thus, one can find a measurable partition $\cF$ formed by random local unstable disks generated by points from $\cH^G_{\lambda_0}$, whose union possess the positive $\mu^G$ measure and $\mu_n^G$ measure ($n$ large enough), by constructing a foliated chart in lift space $\Omega^{\ZZ}\times M$, this is a consequence of Proposition \ref{forc}. Since $\mu_n^G$ has positive Lyapunov exponents by assumption for all large $n$, it is a random Gibbs cu-state guaranteed by Lemma \ref{rGibbs}, this together with the uniform behavior on local unstable disks of $\cB$, one get that the densities of disintegrations of these Gibbs cu-states along $\cB$ are bounded from above and from below (Lemma \ref{fd}). By passing to the limit, one can get the absolutely continuity of $\mu$ along a measurable partition $\cF$ formed by local unstable disks, where $\cF$ is obtained from $\cB$ automatically.
\smallskip
\paragraph{\textbf{Settings of this section}.}
Throughout this section, let $\Lambda$ be an attractor with dominated splitting $T_{\Lambda}M=E\oplus_{\succ}F$. Let us take $\varepsilon_0$ as in Lemma \ref{randomconti} and thus, one can fix constant $a>0$ such that $\rho\left(E(\underline{\omega},x),E(x)\right)\le a$ for every $(\underline{\omega},x)\in \Lambda_{\varepsilon_0}$. Moreover, we assume that all the noise levels considered in this section are bounded from above by $\varepsilon_0$.
%Since the distributions $E$ and $F$ are continuous, we may extend them continuously to some trapping region $U$ of $\Lambda$, denoted by $E$ and $F$ as well. We shall also fix small constants $\varepsilon_0>0$ and $a>0$ such that:
%\begin{itemize}
%\item[---] All the noise levels considered here are bounded from above by $\varepsilon_0$.
%\item[---] For every random perturbation $\mathscr{R}_f^\varepsilon$, $\varepsilon\le \varepsilon_0$, $\Lambda_{\varepsilon}\subset \mathcal{N}_{\varepsilon_0}^{\ZZ}\times U$.
%\item[---] We have $d(E(\underline{\omega},x),E(x))\le a$, for every $(\underline{\omega}, x)\in \Lambda_{\varepsilon}$ and $\varepsilon\le \varepsilon_0$.
%\end{itemize}

%\begin{Definition}\label{def:abs}
%Given a Borel probability $\mu$,
%let $\cP$ be a measurable partition formed by disjoint local unstable disks, whose union admits positive $\mu$ measure. We say $\mu$ has \emph{absolutely continuous conditional measures} along $\cP$, if $\mu_{\gamma}\ll {\rm Leb}_{\gamma}$ for $\hat{\mu}$ almost every $\gamma\in \cP$.
%\end{Definition}

\subsection{Unstable disks at backward contracting points}\label{se41}
%Now we recall the notion of (random) local unstable disk.
%Use $W^{E,u}_{\delta}(\underline{\omega},x)$ to denote the local random unstable disk of $(\underline{\omega},x)$ with radius $\delta$ around $x$, which is tangent to the random bundle $E(\underline{\omega},\cdot)$ and $(C,\lambda)$-\emph{backward contracting} for some $C>0, \lambda\in (0,1)$, in the sense that
%$$
%d(f_{\underline{\omega}}^{-n}(y), f_{\underline{\omega}}^{-n}(z))\le C\lambda^n d(y,z),\quad \textrm{for every}~n\in \NN,
%$$
%whenever $y,z \in W_{\delta}^{E,u}(\underline{\omega},x)$.
%Similarly, we represent $W^{E,u}_{\delta}(x)$ the local unstable disk of $x$ with radius $\delta$ around $x$, which is tangent to $E$ and $(C,\lambda)$-\emph{backward contracting} (for some $C>0$ and $\lambda\in (0,1)$), i.e.,
%$$
%d(f^{-n}(y), f^{-n}(z))\le C\lambda^n d(y,z),\quad \textrm{for every}~n\in \NN,
%$$
%whenever $y,z \in W_{\delta}^{E,u}(x)$.
%
%Sometimes, one need to look random unstable disks on the product space $\Omega^{\ZZ}\times M$. For this,
%we use $\{\underline{\omega}\}\times W^{E,u}_{\delta}(\underline{\omega},x)\subset M_{\underline{\omega}}$ to denote the endowed random local unstable disk associated to $W^{E,u}_{\delta}(\underline{\omega},x)$.

Given $\lambda \in (0,1)$, consider the following ``Pesin block":
$$
\cH^G_{\lambda}=\left\{(\underline{\omega},x):  \prod_{i=0}^{n-1} \|Df^{-1}_{\sigma^{-i}\underline{\omega}}|_{E(G^{-i}(\underline{\omega},x))}\|\le \lambda^{n},\quad \forall n\in\NN \right\}.
$$
Analogously, we define
$$
\mathcal{H}_{\lambda}=\left\{x:  \prod_{i=0}^{n-1}\|Df^{-1}|_{E(f^{-i}x)}\| \le \lambda^n, \quad \forall n\in \NN \right\}.
$$

\begin{Definition}\label{bck}
Given $\delta>0$, $\lambda\in (0,1)$,
we say $(\underline{\omega},x)$ is $\lambda$-\emph{backward contracting} if $(\underline{\omega},x)\in \cH^G_{\lambda}$. Similarly, we also say a point $x$ is $\lambda$-\emph{backward contracting} if it is contained in $\cH_{\lambda}$.
\end{Definition}

The lemmas below stress that any backward contracting points inherits the (random) local unstable disk, it is a folklore result which can be deduced from the Plaque family theorem \cite{HPS77}, see \cite[Theorem 4.7]{AB12} for a direct proof.

%we consider the following compact set
%$$
%\mathscr{H}^G_{\lambda}=\left\{(\underline{\omega},x): \prod_{i=0}^{n-1} \|Df^{-1}_{\sigma^{-i}\underline{\omega}}|_{E(G^{-i}(\underline{\omega},x))}\|\le \lambda^{n}, \quad \forall n\in \mathbb{N}\right\}.
%$$

\begin{Lemma}\label{lu}
Given $\lambda\in (0,1)$, there exist $\sigma \in (\lambda, 1)$ and $\delta=\delta(\lambda)>0$, such that if $(\underline{\omega},x)\in \cH_{\lambda}^G$,
%$$
%\prod_{i=0}^{n-1} \|Df^{-1}_{\sigma^{-i}\underline{\omega}}|_{E(G^{-i}(\underline{\omega},x))}\|\le \lambda^{n} \quad \textrm{for any}~n\in \NN,
%$$
then there exists a random local  unstable disk $W_{\delta}^{E,u}(\underline{\omega},x)$ of $(\underline{\omega},x)$ that is $\sigma$-backward contracting.
\end{Lemma}

We have deterministic version of Lemma \ref{lu} as follows:
\begin{Lemma}\label{cu}
Given $\lambda\in (0,1)$, there exist $\sigma \in (\lambda, 1)$ and $\delta=\delta(\lambda)>0$,  such that if $x\in \cH_{\lambda}$, then there exists a local unstable disk $W^{E,u}_{\delta}(x)$ of $x$ which is $\sigma$-backward contracting.
\end{Lemma}

%This result can be deduced from the Plaque family theorem \cite{HPS77} (see also \cite{ABC11}), a much more direct proof can be found in \cite[Theorem 4.7]{AB12} with a modification to random case.
%In fact, when reducing to the deterministic dynamical systems, we have the following:
%\begin{Corollary}
%Given $0<\lambda<1$, there exists $\sigma \in (\lambda, 1)$, $\delta=\delta(\lambda)>0$ such that if $x$ is $\lambda$-backward contracting, then there exists the local unstable manifold $W^{E,u}_{\delta}(x)$ at $x$ with radius $\delta$ around $x$, for which
%$$
%d(f^{-n}y, f^{-n}z)\le \sigma ^n d(y,z) \quad \textrm{for every}~n\in \mathbb{N}
%$$
%whenever $y,z\in W^{E,u}_{\delta}(x)$.
%Moreover, it is tangent to $E$ everywhere.
%\end{Corollary}

%(see also \cite{ABC11}).

\subsection{A foliated chart of random unstable disks}\label{se42}
Denote by $\DD^k$ the $k$-dimensional compact unit ball of $\RR^k$.

\begin{Definition}\label{Def:foliated-chart}
Let $\cK$ be a compact metric space. A \emph{foliated chart} associated to $\cK$ is a homeomorphism $\Phi:~\cK \times \DD^{k}\rightarrow \cB$ such that
\begin{itemize}

\item $\Phi_p=\Phi|\{p\}\times \DD^k$ is a diffeomorphism for each $p\in \cK$.

\item $\Phi_p$ maps $\DD^k$ to disjoint (endowed random) local unstable disks with dimension $k$.
\end{itemize}
\end{Definition}

We will also say the above $\Phi$ is a foliated chart of (endowed random) local unstable disks.
It follows from Definition \ref{Def:foliated-chart} that
any foliated chart induces a measurable partition formed by disjoint (endowed random) local unstable disks. This suggests that one can also describe the absolute continuity via foliated charts.

More precisely, let $\mu$ be a Borel probability, one gets that $\mu$ has absolutely continuous conditional measures along the measurable partition induced by a foliated chart $\Phi$, if and only if the pullback $\nu:=\Phi^{\ast}\mu=\mu \circ \Phi$ has the absolutely continuous conditional measures along the ``vertical lines" $\{\{x\}\times \DD^k: x\in \cK\}$, which means that there exists a measurable function $\rho:\cK\times \DD^k\longrightarrow [0,\infty)$ such that
$$
\nu(A)=\int_A \rho(x,y){\rm d}{\rm Leb}_{\DD^k}(y){\rm d}\hat{\nu}(x)
$$
for every measurable subset $A$ of $\cK\times \DD^k$, where $\hat{\nu}$ is the quotient measure of $\nu$ w.r.t. ``vertical lines" defined by $\hat{\nu}(\xi)=\nu(\xi \times \DD^k)$ for measurable $\xi \subset \cK$.

We will use the following argument (see e.g. \cite[Proposition 7.3]{v99}).
\begin{Lemma}\label{vp}Let $\cK$ be a compact metric space and $\nu$ a measure on product space $\cK\times \DD^k$, if there exists $C>0$ such that for every open subset $\xi\subset \cK$ satisfying $\hat{\nu}(\partial{\xi})=0$ and open subset $\eta \subset \DD^k$ one has
$$
\nu(\xi\times \eta)\le C\cdot {\rm Leb}_{\DD^k}(\eta)\hat{\nu}(\xi),
$$
then $\nu$ has absolutely continuous conditional measures along the vertical lines with density bounded from above by $C$.
\end{Lemma}

%The main idea to achieve Theorem \ref{bac} , our strategy is to construct a foliated chart in lift space $\Omega^{\ZZ}\times M$ of endowed random local unstable disks
%
%of endowed random local unstable disks.
%As a central step to achieve Theorem \ref{bac}, we shall construct a foliated chart, which will induce the measurable partition $\cF$ formed by local unstable disks we needed. For this, one construct a foliated chart in lift space $\Omega^{\ZZ}\times M$ of endowed random local unstable disks, whose image under the chart exhibits not only positive $\mu^G$ measure, but also positive $\mu^G_n$ measures for all large $n\in \NN$, where $\mu_n$ are ergodic stationary measures such that $\mu_n\to \mu$, and $\mu_n^G, \mu^G$ are the corresponding $G$-invariant measures as stated in Lemma \ref{Lem:invariant-measure-G}.  Actually, we prove the following general result.

\begin{Proposition}\label{forc}
Let $f$ be a $C^2$ diffeomorphism with an attractor $\Lambda$ admitting dominated splitting $T_{\Lambda} M=E\oplus_{\succ} F$. Given $0<\lambda_0<1$, assume that $\{\mu_n\}_{n\in \mathbb{N}}$ is a sequence of stationary measures for $\nu_n$ so that $\liminf_{n\to \infty}\mu_n^G(\cH^G_{\lambda_0})>0$, and $\lim_{n\to +\infty}\mu_n=\mu$ for some stationary measure $\mu$.

Then there exist $0<\delta<\delta_u$, and a foliated chart $\Phi: \cK\times \DD^{{\rm dim}E} \rightarrow \cB$ such that
\begin{itemize}
\item $\mu^G(\partial(\cB))=0$, $\mu^G(\cB)>0$ and $\mu_n^G(\cB)>0$ for all sufficiently large $n$.
\item each $\Phi_{(\underline{\omega},x)}(\DD^{{\rm dim}E} )=\{\underline{\omega}\}\times \gamma_{\delta}^u(\underline{\omega},x)$, where $\gamma_{\delta}^u(\underline{\omega},x)\subset W^{E,u}_{\delta_u}(\underline{\omega},y)$ is a random local unstable disk with radius $\delta$ around $x$, for some $(\underline{\omega},y)\in \cH^G_{\lambda_0}$.
\end{itemize}
\end{Proposition}

\begin{proof}
Suppose that $\cN_{0}:= {\rm supp \nu_{\varepsilon_0}}$, then we have ${\rm supp}(\nu_n)\subset \cN_{0}$ for any $n\in \NN$ by our setting on random perturbations.
By Lemma \ref{lu}, we take $\delta_u$ as the radius of random local unstable disks of points in $\cH^G_{\lambda_0}$.
%Now we introduce the following typical cylinder:
%For fixed cone width $a$, for
%$r>0$, take a $C^1$ closed disk $\gamma$ of radius $r$ at some point $x$ and tangent to cone field $K_a^E$, define
%$$
%\mathcal{C}_{r}=\bigcup_{y\in \gamma}\left\{\exp_y(v): v\perp T_y\gamma,\|v\|\le r \right\}.
%$$
%Then, $\mathcal{C}_{r}$ is a tubular neighborhood with base disk $\gamma$.
%%Observe that it is a diffeomorphic image of $\DD^E\times \DD^F$, and thus it is a cylinder.
%It follows from the continuity of sub-bundle $E$, there exists a universal $\tau>1$ such that for any $r_0\ll \delta_u/\tau$, one has that
%\begin{itemize}
%\item $\mathcal{C}_{r_0}$ contains the ball $B(x,\tau^{-1}r_0)$ with center $x$;
%\item $2\tau^{-1}r_0\le {\rm diam}(\gamma)\le 2\tau r_0$, for any $\gamma$ crossing $\mathcal{C}_{r_0}$.
%\end{itemize}
Since $\mu_n\to \mu$, we have $\mu_n^G \to \mu^G$ by Lemma \ref{Lem:invariant-measure-G}, together with the compactness of $\cH^G_{\lambda_0}$ and assumption $\liminf_{n\to \infty}\mu_n^G(\cH^G_{\lambda_0}):=\alpha_0>0$ yield that $\mu^G(\cH^G_{\lambda_0})\ge \alpha_0$. Choose any $(\underline{\omega_0}, x_0)$ in the supported of $\mu^G|\cH^G_{\lambda_0}$, hence for any $r\ll \delta_u$, one has $\mu^G(\cN_{0}^{\ZZ}\times B(x_0,r))>0$, and thus $\mu^G(\cN_{0}^{\ZZ}\times \overline{B(x_0,r)})>0$. Rewrite $C_r=\cN_{0}^{\ZZ}\times \overline{B(x_0,r)}$, it is compact.
Up to reducing $r$, one can take a smooth compact disk $\Delta^F$ such that for every $x\in \overline{B(x_0, r_0)}$, if $N_x$ is a smooth disk tangent to $\mathcal{E}_a$ with radius $\delta_u$ around $x$, then $N_x$ transverse to $\Delta^F$ at a single point.
This implies that for every $(\underline{\omega},x)\in \cH^G_{\lambda_0}\cap C_r$, the random local unstable disk $W^{E,u}_{\delta_u}(\underline{\omega},x)$ transverse to $\Delta^F$ at a point, use $\cL$ to denote the family of these random local unstable manifolds. More precisely, we take
$$
\cL=\left\{ \gamma(\underline{\omega},x)  : \gamma(\underline{\omega},x)=W^{E,u}_{\delta_u}(\underline{\omega},x),~\textrm{for some}~(\underline{\omega}, x)\in  \cH^G_{\lambda_0}\cap C_r\right\}.
$$
Then, consider the following subset of $\cN_0^{\ZZ}\times \Delta^F$,
$$
\cK=\left\{(\underline{\omega},x)\in \cN_0^{\ZZ}\times \Delta^F: x=\gamma(\underline{\omega},y)\cap \Delta^F, ~ \textrm{for some}~\gamma(\underline{\omega}, y)\in \cL \right\}.
$$

\begin{Lemma}
$\cK$ is a compact subset of $\cN_{0}^{\ZZ}\times \Delta^F$.
\end{Lemma}
\begin{proof}
%Let us denote by $\cL$ the family of random unstable disks of points in $\Lambda^G(E,\lambda_0)\cap C_k$. More precisely, define
%$$
%\cL=\left\{ \gamma(\underline{\omega},x)  : \gamma(\underline{\omega},x)=W^{E,u}_{\delta_u}(\underline{\omega},x) ~\textrm{for some}~(\underline{\omega}, x)\in  \Lambda^G(E,\lambda_0)\cap C_k\right\}.
%$$
Let $(\underline{\omega_m},x_m)\in \cK$ be a sequence of points which converges to a point $(\underline{\omega},x)\in \cN_{0}^{\ZZ}\times M$ as $n\to +\infty$. One knows that $x$ is contained in $\Delta^F$, as $\Delta^F$ is compact, thus it suffices to show that there is $\gamma(\underline{\omega},y)\in \cL$ containing $x$. By construction of $\cK$, for each $(\underline{\omega_m},x_m)\in \cK$, there exists a point $y_m$ so that $x_m\in \gamma(\underline{\omega_m},y_m)\in \cL$. Considering subsequences if necessary, we assume that $(\underline{\omega_m},y_m)$ converges to a point $(\underline{\omega} ,y)$, using Ascoli-Arzela theorem,
there is a smooth disk $\gamma$ containing $y$
such that  $\gamma(\underline{\omega_m},y_m)$ converges to $\gamma$ in $C^1$-topology. This also implies that $\gamma$ is tangent to $E(\underline{\omega}, \cdot)$ everywhere.
Since  $\cH^G_{\lambda_0}\cap C_r$ is compact, we have $(\underline{\omega},y)\in \cH^G_{\lambda_0}\cap C_r$, and so $\gamma \in \cL$.
Now we deduce that $x\in \gamma$, indeed, for any $\varepsilon>0$, there exists $m$ large enough such that $d(x_m,x)<\varepsilon/2$, and some point $x_{\varepsilon}\in \gamma$ satisfying $d(x_m, x_{\varepsilon})<\varepsilon/2$, since $\gamma(\underline{\omega_m},y_m)$ converges to $\gamma$. Consequently, $d(x,x_{\varepsilon})<\varepsilon$, by compactness of $\gamma$ we obtain $x\in \gamma$.
\end{proof}

It follows from the construction of $\cK$ that there exists $\delta \le \delta_u$ and a family of random unstable disks $\cF_{\delta}=\{\gamma^u_{\delta}(\underline{\omega},x): (\underline{\omega},x )\in \cK\}$ such that for each $(\underline{\omega},x)\in \cK$,
\begin{itemize}
\item there exists $\gamma(\underline{\omega}, y)\in \cL$ such that $\gamma^u_{\delta}(\underline{\omega},x)\subset \gamma(\underline{\omega},x)$ and has radius $\delta$ with center $x$;
\item the intersection of $\gamma^u_{\delta}(\underline{\omega},x)$ with $C_r$ is contained in the interior of $\gamma^u_{\delta}(\underline{\omega},x)$.
\end{itemize}
Furthermore, as the random local unstable disks of $\mathcal{F}_{\delta}$ depend continuously on the points of $\cK$, there is a continuous map
$\phi: \cK \longrightarrow {\rm Emb}^{1}(\DD^{{\rm dim}E} , \Omega \times M)$ such that $\phi(\underline{\omega},x)(\DD^E)=\{\underline{\omega}\}\times \gamma^u_{\delta}(\underline{\omega},x)$.  As a result, one can define the map
\begin{align*}
\Phi: \cK \times \DD^{{\rm dim}E}  & ~~\longrightarrow ~~\Omega^{\mathbb{Z}}\times M \\
(\underline{\omega}~,~x~ ,~y) &~~\longmapsto ~~\phi(\underline{\omega},x)(y).
\end{align*}
Let $\cB=\Phi( \cK \times \DD^{{\rm dim}E}  )$, then $\Phi: \cK \times \DD^{{\rm dim}E} \longrightarrow \cB$ is a foliated chart.
%then by taking the base disk $\gamma=W^{E,u}_{r_0}(\underline{\omega_c},x_c)$, we construct the tubular neighborhood $\cC_{r_0}$ around $x_c$ as above. By the choices of $r_0$ and $r_1$, we get
%$$
%B(x_k,r_1)\subset B(x_c, 2r_1) \subset B(x_c, \tau^{-1}r_0)\subset \cC_{r_0}.
%$$
Since $\cB$ consisted of all the endowed disks from $\mathscr{F}_{\delta}$, thus $\mu^G(\cB)\ge \mu^G(\cH^G_{\lambda_0}\cap C_r)>0$. Reducing $r$, $\delta$ if necessary, one can assume $\mu^G(\partial(\cB))=0$. Then $\mu_n^G\to \mu^G$ gives
$$
\lim_{n\to \infty}\mu_n^G(\cB)=\mu^G(\cB)>0,
$$
which ends the proof of Proposition \ref{forc}.
\end{proof}

\subsection{Existence of backward contracting points}\label{se43}
The goal of this subsection is to show the existence of enough backward contracting points in measure-theoretic sense, under the assumption that the randomly ergodic limit has positive integration Lyapunov exponents along $E$.  More precisely, we have

\begin{Theorem}\label{main2}
Let $f$ be a $C^2$ diffeomorphism with an attractor $\Lambda$ admitting a dominated splitting $T_{\Lambda} M=E\oplus_{\succ} F$. If $\mu$
is a randomly ergodic limit such that
$
\lambda_E(\mu)>0,
$
then there exists $\lambda_0, \alpha_0 \in (0,1)$ with following properties:
\begin{itemize}
\item $\mu^G(\cH^G_{\lambda_0})\ge \alpha_0$ and $\mu(\cH_{\lambda_0})\ge \alpha_0$.
\item If $\{\mu_n\}_{n\in \mathbb{N}}$ is the sequence of ergodic stationary measures such that $\lim\limits_{n\to \infty}\mu_n=\mu$,
then $\liminf\limits_{n\to \infty}\mu_n^G(\cH^G_{\lambda_0})\ge \alpha_0.$
\end{itemize}

\end{Theorem}

Recall the well known Pliss Lemma, one can see a proof in \cite[Lemma 3.1]{ABV00}.

\begin{Lemma}\label{Plisslemma}
Given constants $p_0\ge p_1> p_2>0$, there exists $\rho=\rho(p_0,p_1,p_2)>0$ such that for any integer $N\in \mathbb{N}$,  and  real numbers $a_1,\cdots,a_N$, if
$$
\frac{1}{N}\sum_{i=1}^{N}a_i\ge p_1, ~~a_i\le p_0 ~~~\text{for every} ~~1\le i \le N.
$$
Then there is an integer $\ell>\rho N$ and a sequence $1\le n_1 \le \cdots \le n_{\ell} \le N$ such that
$$
\frac{1}{n_j-n}\sum_{i=n+1}^{n_j}a_i\ge p_2 ~~~\text{for any}~~0\le n<n_j, ~~~\text{and}~~1\le j \le \ell.
$$
\end{Lemma}

%As a corollary of the Pliss Lemma, we have
%\begin{Corollary}\label{hyperbolictimediffe}
%Given $\tau_1>\tau_2>0$, there is $\rho:=\rho(\tau_1,\tau_2,f)>0$ such that,  if $(\underline{\omega},x)$ is $\tau_1$-random non-uniformly expanding,
%then for any sufficiently large $N\in \mathbb{N}$, there are  $\tau_2$-random hyperbolic times $1\le n_1<\cdots<n_j\le N$ for $(\underline{\omega},x)$, with $j\ge \rho N$.
%\end{Corollary}

We need the following lemma, see a proof in \cite[Lemma 2.1]{MCYMZ}.

\begin{Lemma}\label{Lem:MCYPliss}
Let $\{a_n\}_{n=1}^{\infty}$ be a sequence of real numbers, for $N\in \NN$, if
$$
\sum_{i=1}^{n}a_i \le 0,~~~~ \forall~n\ge N.
$$
Then there exists $1\leq k \le N$ such that
$$
\sum_{i=1}^n a_{i+k-1} \le 0,~~~~\forall~n\in \NN.
$$
\end{Lemma}

\begin{Definition}
Given subset $\LL\subset \NN$ we define the \emph{density} of $L$ as
$$
\limsup_{n\to +\infty}\frac{\#([1,n]\cap \LL)}{n}.
$$

\end{Definition}

By Pliss Lemma (Lemma \ref{Plisslemma}) and Lemma \ref{Lem:MCYPliss}, one can get the following technical result.

\begin{Proposition}\label{Lem:Pliss-infinite}

Given constants $p_0\ge p_1> p_2>0$, there exists $\rho=\rho(p_0,p_1,p_2)>0$  such that for real numbers $a_1,a_2,\cdots,a_n,\cdots$, if
$$\limsup_{n\to\infty}\frac{1}{n}\sum_{i=1}^n a_i\le p_2, ~~|a_i|\le p_0,$$
Then there is a subset $K\subset\NN$ whose density is larger than $\rho$ such that for any $k\in K $, we have
$$\sum_{i=0}^{n-1}a_{i+k}\le np_1,~~~\forall n\in\NN.$$

\end{Proposition}

\begin{proof}
In what follows, we shall show that there exists a sequence of integers $\{\ell_m\}_{m\in \mathbb{N}}$ such that $\ell_0=1$, $\ell_m\to \infty$, and $I_m\subset [\ell_{m-1}, \ell_m]$ such that every $k\in I_m, m\in \mathbb{N}$,
\begin{equation}\label{P0}
\sum_{i=0}^{n-1}a_{i+k}\le np_1, ~~\forall n\in \mathbb{N}.
\end{equation}
Moreover, they admit the following property:
\begin{equation}
\frac{\#(I_1\cup \cdots I_m)}{\ell_m}\ge \rho
\end{equation}
for some $\rho$ depending only on $p_0,p_1,p_2$. Therefore, by taking $K =\bigcup_{m\in \mathbb{N}}I_m$, we then have
$$
\limsup_{n\rightarrow \infty}\frac{1}{ n}\#\left([1,n]\cap K]\right){n}\ge \rho.
$$
to end the proof.

Let us fix $p=(p_1+p_2)/2$.  By assumption, for any $j\in \mathbb{N}$ we have
$$
\limsup_{n\to \infty}\frac{1}{n}\sum_{i=0}^{n-1}a_{i+j}\le p_2.
$$
Thus, there exists $N(j)\in\mathbb{N}$ such that
\begin{equation}\label{P1}
\frac{1}{n}\sum_{i=0}^{n-1}a_{i+j}\le p, ~~\forall n\ge N(j).
\end{equation}
By applying Lemma \ref{Lem:MCYPliss}, we know that there exists $j\le \ell(j) \le N(j)+j$ so that
\begin{equation}\label{P2}
\frac{1}{n}\sum_{i=0}^{n-1}a_{i+\ell(j)}\le p,~~\forall n\ge 1.
\end{equation}

Take $j_1$ large enough such that $\ell_1:=\ell(j_1)\ge j_1 \ge N(1)$, then by (\ref{P1}) we get
\begin{equation}\label{P3}
\frac{1}{\ell_1}\sum_{i=1}^{\ell_1}a_i\le p.
\end{equation}
By applying Lemma \ref{Plisslemma}, there exists $\rho=\rho(p_0,p_1,p_2)\in (0,1)$ so that, there is a sequence of integers $I_1=\{k_1, \cdots, k_{j(1)}\}\subset \{1,\cdots \ell_1\}$ satisfying $j(1)\ge \rho \ell_1$, and for each $k\in I_1$, one has
\begin{equation}\label{P4}
\sum_{i=0}^{m-1}a_{i+k}\le mp_1,~~\forall 1\le m \le  \ell_1-k+1.
\end{equation}
It follows from (\ref{P2}) and (\ref{P4}) that for $m>\ell_1-k+1$ we have
\begin{eqnarray*}
\sum_{i=0}^{m-1}a_{i+k}&\le & \sum_{i=0}^{\ell_1-k}a_{i+k} +\sum_{i=0}^{m+k-\ell_1-1}a_{\ell_1+i}\\
&\le &(\ell_1-k)p_1+(m+k-\ell_1-1)p\\
&\le &mp_1.
\end{eqnarray*}
This together with (\ref{P4}) ensure that each $k\in I_1$ has property (\ref{P0}).

Now we take $j_2=\ell_1$, and then choose $\ell_2=\ell(j_2)$ with properties (\ref{P2}) and (\ref{P3}) by letting $j=j_2$ and replacing $\ell_1$ with $\ell_2$, respectively. Again, by applying Lemma \ref{Plisslemma} for $\ell_2$, one can find a sequence of integers $I_2\subset [\ell_1, \ell_2]$ such that
\begin{itemize}
\item each $k\in I_2$ exhibits the property (\ref{P4}) for $\ell_2$, i.e.,
\begin{equation}\label{j2}
\sum_{i=0}^{m-1}a_{k+i}\le mp_1, ~~\forall 1\le m \le \ell_2-k+1.
\end{equation}
\item $\#\{I_1\cup I_2\}\ge \rho \ell_2$.
\end{itemize}
Here we use the fact that each $k\le \ell_1$ satisfies (\ref{P4}) for $\ell_2$ whenever it satisfies (\ref{P4}) for $\ell_1$.
Combining (\ref{j2}) with (\ref{P2}) for $j_2$, one can conclude that each $k$ in $I_2$ satisfies the property (\ref{P0}), as in $\ell_1$ case.

By taking $j_k=\ell_{k-1}$ and $\ell_k=\ell(j_k)$, $k\ge 2$ we can select a sequence of integers $\ell_m$, $m\in \mathbb{N}$ and $I_m \subset [\ell_{m-1}, \ell_m]$ such that each $k\in \cup_{m\in \mathbb{N}}I_m$ satisfies property (\ref{P0}), with density estimates (\ref{P1}) for $\rho(p_0, p_1, p_2)$.
\end{proof}

Now we can give the proof of Theorem \ref{main2} by using Proposition \ref{Lem:Pliss-infinite}.

\begin{proof}[Proof of Theorem \ref{main2}]
One observes that the first item can be deduced from the second item by taking into the account the compactness of  $\cH^G_{\lambda_0}$.  In fact, if we take ergodic stationary measures $\mu_n$ converges to $\mu$, then $\mu_n^G$ converges to $\mu^G=\delta_{\omega_f}^{\ZZ}\times \mu$ by applying Corollary \ref{toinv}, and thus
$$
\mu^G(\cH^G_{\lambda_0})\ge \limsup_{n\to \infty}\mu^G_n(\cH^G_{\lambda_0})\ge \liminf_{n\to \infty}\mu^G_n(\cH^G_{\lambda_0}) \ge \alpha_0.
$$
Moreover, we have the following relation
\begin{eqnarray*}
\mu^G(\cH^G_{\lambda_0})&=& \int \int \chi_{\cH^G_{\lambda_0}}(\underline{\omega},x){\rm d} \delta_{\omega_f}^{\ZZ}(\underline{\omega}){\rm d}\mu(x)\\
&=&  \int \chi_{\cH^G_{\lambda_0}}(\underline{\omega^f},x){\rm d}\mu(x)\\
&=&  \int \chi_{\cH_{\lambda_0}}(x){\rm d}\mu(x)\\
&=&\mu (\cH_{\lambda_0}).
\end{eqnarray*}
Therefore, $\mu (\cH_{\lambda_0})\ge \alpha_0$. Now it suffices to give the proof of the second item.

Put
$$
\alpha=\int \log \|Df^{-1}|_{E(x)}\|^{-1}{\rm d}\mu.
$$
By convergence $\mu_n^G\to \delta_{\omega_f}^{\ZZ}\times \mu$ and the continuity of $(\underline{\omega},x)\mapsto \|Df_{\underline{\omega}}^{-1}|_{E(\underline{\omega},x)}\|$, one obtains
\begin{eqnarray*}
\lim_{n\to \infty}\int \log \|Df_{\underline{\omega}}^{-1}|_{E(\underline{\omega},x)}\|{\rm d}\mu_n^G &=& \int  \log \|Df_{\underline{\omega}}^{-1}|_{E(\underline{\omega},x)}\|{\rm d}(\delta_{\omega_f}^{\ZZ}\times \mu) \\
& = & \int \log \|Df^{-1}|_{E(x)}\|{\rm d}\mu\\
&=& -\alpha.
\end{eqnarray*}
Therefore, there exists $n_0\in \NN$ such that for any $m\ge n_0$,
$$
\int \log \|Df_{\underline{\omega}}^{-1}|_{E(\underline{\omega},x)}\|{\rm d}\mu_m^G\le-\alpha/2.
$$
For every $m\ge n_0$,
as $\mu_m^G$ is ergodic, by Birkhoff ergodic theorem, for $\mu_m$-almost every $(\underline\omega,x)$, we have
$$
\lim_{n\to\infty}\frac{1}{n}\sum_{i=1}^{n}\log\|Df^{-1}_{\sigma^{-i}\underline\omega}|_{E(G^{-i}(\underline{\omega},x))}\|=\int \log \|Df_{\underline{\omega}}^{-1}|_{E(\underline{\omega},x)}\|{\rm d}\mu_m^G\le -\alpha/2.
$$
For any fixed $(\underline{\omega},x)$ satisfying above equality,
setting $$a_{i}=\log\|Df^{-1}_{\sigma^{-i}\underline\omega}|_{E(G^{-i}(\underline{\omega},x))}\|, ~~\forall i\in \mathbb{N}$$ and  $$p_0=\max_{(\underline{\omega},x)\in \Lambda_{\varepsilon_0}}\left|\log \|Df^{-1}_{\underline{\omega}}|_{E(\underline{\omega},x)}\|\right|,~~ p_1=-\alpha/4, ~~p_2=-\alpha/2.$$
By Lemma~\ref{Lem:Pliss-infinite},
there is $\rho=\rho(p_0,p_1,p_2)$ independent of $m$ and a subset $K\subset\NN$ such that
\begin{equation}\label{frompliss}
\limsup_{n\to \infty}\frac{1}{n}\#\left([1,n]\cap K\right)\ge \rho.
\end{equation}
and for each  $k\in K$, one has
\begin{equation}\label{contro}
\sum_{i=0}^{n-1}\log \|Df^{-1}_{\sigma^{-k-i}(\underline{\omega})}|_{E(G^{-k-i}(\underline{\omega},x))}\|\le -\frac{\alpha}{4}n,~~\forall n\in \mathbb{N}.
\end{equation}

%\begin{itemize}
%\item for any $k\in K$
%\begin{equation}\label{contro}
%\sum_{i=0}^{n-1}\log \|Df^{-1}_{\sigma^{-k-i}(\underline{\omega})}|G^{-k-i}(\underline{\omega},x)\|\le -\frac{\alpha}{4}n,~~\forall n\in \mathbb{N}.
%\end{equation}
%\item there exists $\rho_0=\rho_0(p_0,p_1,p_2)$ which is independent of $m$ such that,
%\begin{equation}\label{frompliss}
%\limsup_{n\to \infty}\frac{1}{n}\#\left([1,n]\cap K\right)\ge \rho_0.
%\end{equation}
%\end{itemize}
%
%Write
Take $\lambda_0={\rm e}^{-\frac{\alpha_0}{4}}$, property (\ref{contro}) is equivalent to
\begin{equation}\label{tcontr}
\prod_{i=0}^{n-1}\|Df^{-1}_{\sigma^{-k-i}\underline{\omega}}|G^{-k-i}(\underline{\omega},x)\|\le \lambda_0^n,~~\forall n\in \NN, ~\forall k\in K.
\end{equation}
It follows from the definition of $\cH_{\lambda_0}^G$ that $G^{-k}(\underline{\omega},x)$ belongs to $\cH_{\lambda_0}^G$ as long as $k \in K$. Therefore, use the ergodicity of  $\mu_m^G$ again, one gets
\begin{eqnarray*}
\mu_m^G(\cH^G_{\lambda_0})&=& \lim_{n\rightarrow \infty}\frac{1}{n}\sum_{i=1}^n\chi_{\cH_{\lambda_0}^G}(G^{-i}(\underline{\omega},x))\\
&=&\lim_{n\to\infty}\frac{1}{n} \#\left\{i:~1\le i\le n,~G^{-i}(\underline\omega,x)\in \cH^G_{\lambda_0}\right\}\\
&=&\limsup_{n\to \infty}\frac{1}{n}\#{\left([1,n]\cap K \right)}
\end{eqnarray*}
By applying the density estimate (\ref{frompliss}) and taking $\alpha_0=\rho$, we get $\mu^G_m(\cH^G_{\lambda_0})\ge \alpha_0$ to complete the proof of the second item.
\end{proof}

\subsection{Proof of Theorem \ref{bac}}\label{se44}
In light of above preparations, we can complete the proof of Theorem \ref{bac}.

\begin{proof}[Proof of Theorem \ref{bac}]
Under the assumption of Theorem \ref{bac}, assume that $\{\mu_n\}$ is a sequence of ergodic stationary measures converges to $\mu$ in weak$^{\ast}$ topology. By Theorem \ref{main2}, there exist $\lambda_0, \alpha_0 \in (0,1)$ such that $\liminf_{n\to \infty}\mu_n^G(\cH^G_{\lambda_0})\ge \alpha_0$, recall that $\{\mu_n^G\}$ are $G$-invariant measures defined by Lemma \ref{Lem:invariant-measure-G}. By applying Proposition \ref{forc}, one can built a foliated chart $\Phi: \cK\times \DD^{{\rm dim}E}\longrightarrow \cB$ of random local unstable disks generated by points from $\cH^G_{\lambda_0}$ such that $\mu^G(\cB)>0$ and there exists $N\in \NN$ satisfying $\mu_n^G(\cB)>0$ for any $n\ge N$. More precisely, the construction of $\cB$ suggests that there exist $\delta>0, \sigma\in (0,1)$ and a family of random local unstable disks $\mathcal{F}_{\delta}=\{\gamma^u_{\delta}(\underline{\omega},x): (\underline{\omega},x)\in \cK\}$ with following properties:
\begin{itemize}
\item $\gamma_{\delta}^u(\underline{\omega},x)$ is a random local unstable disk of radius $\delta$ around $x$;

%\item every $\gamma^u_{\delta}(\underline{\omega},x)$ is contained in $W^{E,u}_{\delta_u}(\underline{\omega},y_x)$ for some $(\underline{\omega},y_x)\in \cH^G_{\lambda_0}$;

\item each element of $\mathcal{F}_{\delta}$ is $\sigma$-backward contracting;

\item $\cB$ is the union of the endowed disks from $\mathcal{F}_{\delta}$, i.e.,
$$
\cB=\bigcup_{\gamma_{\delta}^u(\underline{\omega},x)\in \mathcal{F}_{\delta}}\{\underline{\omega}\}\times \gamma_{\delta}^u(\underline{\omega},x);
$$
\item $\mu^G(\partial(\cB))=0$.
\end{itemize}
Define $\cF^G=\{\{\underline{\omega}\}\times \gamma_{\delta}^u(\underline{\omega},x):\gamma_{\delta}^u(\underline{\omega},x)\in\mathcal{F}_{\delta} \}$.
The assumption $\lambda_E(\mu)>0$ implies that the Lyapunov exponents of $\mu_n^G$ along $E$ are positive for all $n\ge N$, up to increasing $N$. By Lemma \ref{rGibbs}, $\{\mu_n^G\}$ are random Gibbs $cu$-state associated to $E$. Therefore, according to Lemma \ref{fde}, the densities of the conditional measures of $\mu_n^G$ along $\cF^G$ are bounded from above by a constant $C$ independent of $n$. This implies that when pulling back to $\cK \times \DD^{{\rm dim}E}$ by foliated chart $\Phi$, and denote by $\nu_n=\mu_n^G \circ \Phi $ for every $n\ge N$, there is a constant $C_{\Phi}$ such that for every measurable subset $\xi\subset \cK$ and $\zeta \subset \DD^{{\rm dim}E}$ we have
\begin{equation}\label{fdcd}
\nu_n(\xi \times \zeta)\le C_{\Phi} \cdot {\rm Leb}_{\DD^{ {\rm dim}E}}(\zeta) \hat{\nu}_n(\xi)\quad \textrm{for every}~n\ge N.
\end{equation}
Define $\nu=\mu^G \circ \Phi $, then the convergence $\mu_n^G\to \mu^G$ together with $\mu^G(\partial(\cB))=0$ gives $\nu_n \to \nu$ and $\hat{\nu}_n\to \hat{\nu}$ as $n\to \infty$. Take any open subset $\zeta \subset \DD^{{\rm dim}E}$ and open subset $\xi \subset \cK$ with $\hat{\nu}(\partial{\xi})=0$.  Applying (\ref{fdcd}), one gets
\begin{eqnarray*}
\nu(\xi \times \zeta) &\le& \liminf_{n\to \infty}\nu_n(\xi \times \zeta)\\
&\le & C_{\Phi} {\rm Leb}_{\DD^{{\rm dim}E}}(\zeta)\liminf_{n\to \infty}  \hat{\nu}_n(\xi)\\
&=& C_{\Phi} {\rm Leb}_{\DD^{{\rm dim}E}}(\zeta)\hat{\nu}(\xi)
\end{eqnarray*}
Hence, when applying Lemma \ref{vp} we get that $\mu^G$ has absolutely continuous conditional measures along $\cF^G$. Due to $\mu^G(\cB)>0$ and $\mu^G=\delta_{\omega_f}^{\ZZ}\times \mu$, when denoting $\cF=\{\gamma_{\delta}^u(\underline{\omega},x)\in \mathcal{F}_{\delta}: \underline{\omega}=\underline{\omega^f}\}$, then $\mu$ has positive measure on the union of disks of $\cF$, and moreover
 $\mu$ has absolutely continuous conditional measures along $\cF$. Now we complete the proof of Theorem \ref{bac}.
\end{proof}

\section{Gibbs $cu$-states as ergodic components}\label{se5}

By ergodic decomposition ergodic theorem \cite{Man87}, we assume that $\Sigma$ is the full $\mu$ subset of $M$ such that $\{\mu_x\}_{x\in \Sigma}$ is the family of ergodic decomposition of $\mu$, i.e., $\mu_x$ is ergodic for every $x\in \Sigma$, and moreover, we have
$$
\displaystyle{\int \left(\int\varphi d\mu_x\right)d\mu=\int \varphi d\mu} ~~\text{and}~~\displaystyle{\int \varphi d\mu_x=\lim_{n\rightarrow +\infty}\frac{1}{n}\sum_{i=0}^{n-1}\varphi(f^i(x))}
$$
for every bounded measurable function $\varphi$.

Let $\mathscr{L}$ be the set of \emph{Lyapunov regular} points, where we say $x$ is Lyapunov regular if its Lyapunov exponent of every non-zero vector is well-defined (both forward and backward).
A point $x\in M$ is  \emph{Birkhoff regular} if for any continuous function $\varphi:M\rightarrow \mathbb{R}$, both the forward and backward time averages
$$
\lim_{n\rightarrow +\infty}\frac{1}{n}\sum_{i=0}^{n-1}\varphi(f^i(x))~~~\textrm{and}~~~\lim_{n\rightarrow +\infty}\frac{1}{n}\sum_{i=0}^{n-1}\varphi(f^{-i}(x))
$$
exist and they coincide. Denote by $\mathscr{R}$ the set of Birkhoff regular points of $f$.

It is well known that both $\mathscr{R}$ and $\mathscr{L}$ are $f$-invariant set having full measure for any $f$-invariant probability measure.

By using the backward contraction property of local unstable disks, we have the next result directly.
\begin{Lemma}\label{unstable}
If $\gamma$ is a local unstable disk, then $\mu_x=\mu_y$
for any $x, y\in \gamma \cap \Sigma \cap \mathscr{R}$.
\end{Lemma}

The main goal of this section is to prove the following result.
\begin{Proposition}\label{pe}
Given $0<\sigma <1$ and an $f$-invariant probability measure $\mu$.  Assume that $\cF$ is a measurable partition formed by local unstable disks satisfying:
\begin{itemize}
\item each local unstable disk of $\cF$ is $\sigma$-backward contracting;
\item the union $L$ of the elements of $\cF$ has positive $\mu$ measure;
\item $\mu$ has absolutely continuous conditional measures along $\cF$.
\end{itemize}
Then for $\mu$-almost every $x\in L$, $\mu_x$ is a Gibbs cu-state.
\end{Proposition}

For proving Proposition \ref{pe}, we will use the following result which is essentially contained in \cite[Lemma 6.3]{Man87}. Here we give a proof for completeness.

\begin{Lemma}\label{fromMane} Assume that $\mu$ is an $f$-invariant measure. For a Borel set $B$, if $\mu(B)>0$, then
$
\mu_x(B)>0
$
for $\mu$-almost every $x\in B$.
\end{Lemma}
\begin{proof}
Take $A=B\cap \{x:\mu_x(B)=0\}$, it suffices to show $\mu(A)=0$.
By contradiction, we assume $\mu(A)>0$.
If
$
\mu\big(A\cap\{x:\mu_x(A)>0\}\big)=0,
$
it follows from the definition that
$$
\big\{x:\mu_x(A)>0\big\}=\bigcup_{n\ge 0} f^{-n}\left(A\cap\{x:\mu_x(A)>0\}\right),
$$
and thus $\mu(\{x:\mu_x(A)>0\})=0$. However, ergodic decomposition theorem tell us that
$$
\mu(A)=\int \mu_x(A)d\mu=\int_{\{x:\mu_x(A)>0\}}\mu_x(A)d\mu=0,
$$
which is a contradiction.  Hence, we must have $\mu\left(A\cap\{x:\mu_x(A)>0\}\right)>0$. As a consequence,
$$
\mu(A\cap \{x:\mu_x(B)>0\})\ge \mu(A\cap \{x:\mu_x(A)>0\})>0,
$$
it is contradict to the fact that $A\cap \{x:\mu_x(B)>0\}$ is an empty set by construction.
Therefore, we get $\mu(A)=0$.

%\bigskip
%
%Now we consider the second case, i.e., we have $\mu\left(A\cap\{x:\mu_x(A)>0\}\right)>0$. Since $A\subset B$,
%$$
%\mu(A\cap \{x:\mu_x(B)>0\})\ge \mu(A\cap \{x:\mu_x(A)>0\})>0,
%$$
%it is contradict to the fact that $A\subset \{x:\mu_x(B)=0\}$. The second case is not possible either.
\end{proof}

Recall the following abstract Lemma, see \cite[Lemma 6.2]{ABV00}. %For its proof, one can also see \cite[Lemma ?]{MCY17}.
\begin{Lemma}\label{lem:6.2} Let $\mathcal{M}$ be a measurable space with a finite measure $\eta$.  Let $\mathcal{P}$ be a measurable partition of $\mathcal{M}$ and $\{\mathring{\eta}_P\}_{P\in{\mathcal P}}$ be the conditional measure of $\eta$ w.r.t. $\mathcal P$.

Assume that $\{\eta_x\}_{x\in \mathcal{M}}$ is a family of finite measures on $\mathcal{M}$ with the following properties:
\begin{enumerate}
\item\label{i.measure-decomposation} $x\mapsto \eta_x(A)$ is measurable for every measurable subset $A\subset \mathcal{M}$ and it is constant on each element of $\mathcal{P}$.
\item\label{i.integrable} there exists an integrable function $k: \mathcal{M}\rightarrow [0,\infty)$ such that
$$\eta(A)=\displaystyle{\int k(x)\eta_x(A)d\eta(x)}$$ for any measurable subset $A\subset \mathcal{M}$ .
\item\label{i.coarses-partition} $\{z: \eta_z=\eta_x\}$ is a measurable set with full $\eta_x$-measure, for every $x\in \mathcal{M}$.
\end{enumerate}

Then  for $\eta$-almost every $x \in \mathcal{M}$  and $\widehat{\eta_{x}}$-almost every $P$, for the conditional measures $\{\eta_{x,P}, P\in \mathcal{P} \}$ of $\eta_x$ w.r.t. $\mathcal{P}$, we have
$$
\mathring{\eta}_{P}=\eta_{x,P},$$
where $\widehat{\eta_{x}}$ denotes the quotient measure induced by $\eta_{x}$ on $\mathcal{M}$.

\end{Lemma}

\begin{proof}[Proof of Proposition \ref{pe}]
Since $\mu(L)>0$, by Lemma~\ref{fromMane}, we have that $\mu_x(L)>0$ for $\mu$-almost every $x\in L$.
%The disintegration of $\mu$ with respect to $\cL$ is equivalent to the Lebesgue measures along leaves. By this fact, we have that  almost every point of $L$ is Lyapunov regular.

For every measurable subset $A$ of $L$, by the ergodic decomposition theorem,
$$
\mu(A)=\int \mu_x(A)d\mu(x),
$$
where
$$
\mu_x(A)=\lim_{n\rightarrow+\infty}\frac{1}{n}\sum_{i=0}^{n-1}\chi_{A}(f^{i}(x)).
$$
This suggests that $\mu_x(A)>0$ only if $x$ has some iterate in $A\subset L$. By applying Poincar\'e's recurrence theorem, one can choose $r(x)$ to be the smallest positive integer such that $f^{-r(x)}(x)\in L$ for almost every $x\in L$. The construction of $r(x)$ implies the following fact:
%Then, we conclude that
%\begin{equation}\label{decom}
%\mu(A)=\int_{L} r(x)\mu_x(A)d\mu(x).
%\end{equation}
%Note that here we use the fact: $\mu_x=\mu_{f^i(x)}$ for any $n\in \mathbb{Z}$.
%
%{\color{red}Jun. 4th: One has to check that $r(x)$ is integrable.}
\begin{Claim}
$r(x)$ is integrable on $L$, and
\begin{equation}\label{decom1}
\mu(A)=\int_{L} r(x)\mu_x(A)d\mu(x).
\end{equation}
\end{Claim}
%\begin{proof}
%Let $L_k=\{x\in L: r(x)=k\}$, then $\{f^{-j}(L_k): 0\le j\le k-1, k\ge 1\}$ is the family of disjoint subsets with the following relation:
%\begin{equation}\label{expression}
%\bigcup_{i= 0}^{\infty}f^{-i}(L)=\bigcup_{k=1}^{\infty}\bigcup_{j=0}^{k-1}f^{-j}(L_k)~~~\textrm{mod}(0).
%\end{equation}
%Therefore, we have
%$$
%\mu\left(\bigcup_{i= 0}^{\infty}f^{-i}(L)\right)=\sum_{k=1}^{\infty}\sum_{j=0}^{k-1}\mu(f^{-j}(L_k))=\sum_{k=1}^{\infty}k\mu(L_k)\le 1.
%$$
%Consequently,
%$$
%\int_{L}r(x)d\mu=\int_{\cup_{k=1}^{\infty}L_k}r(x)d\mu=\sum_{k=1}^{\infty}\int_{L_k} k d\mu=\sum_{k=1}^{\infty}k\mu(L_k)\le 1.
%$$
%This shows that $r(x)$ is integrable on $L$.
%
%It remains to show the equality (\ref{decom1}) now.  From
%$$
%\mu_x(A)=\lim_{n\rightarrow+\infty}\frac{1}{n}\sum_{i=0}^{n-1}\chi_{A}(f^{i}(x)),
%$$
%we know that $\mu_x(A)>0$ only if $x$ has some iterate in $A\subset L$. Thus, we obtain
%$$
%\mu(A)=\int \mu_x(A)d\mu(x)=\int _{\bigcup_{i= 0}^{\infty}f^{-i}(L)}\mu_x(A)d\mu(x)
%$$
%It follows form the relation (\ref{expression}) and the $f$-invariance of $\mu$ that
%\begin{eqnarray*}
%\int _{\bigcup_{i= 0}^{\infty}f^{-i}(L)}\mu_x(A)d\mu(x)&=&\sum_{k=1}^{\infty}\int_{\bigcup_{j=0}^{k-1}f^{-j}(L_k)}\mu_x(A)d\mu(x)\\
%&=&\sum_{k=1}^{\infty}\int_{L_k}k\mu_x(A)d\mu(x)\\
%&=&\int_{\bigcup_{k=1}^{\infty}L_k}r(x)\mu_x(A)d\mu(x)\\
%&=&\int_L r(x)\mu_x(A)d\mu(x).
%\end{eqnarray*}
%Consequently, we get (\ref{decom1}).
%\end{proof}}
Now we show that $\mu_x$ has positive Lyapunov exponents along $E$ for $\mu$-almost every $x\in L_{\infty}$. Lemma \ref{fromMane} implies that $\mu_x(L)>0$ for $\mu$-almost every $x\in L$.
Since $\mu_x$-almost every point of $L$ is Lyapunov regular, each $\gamma\in \cF$ is $\sigma$-backward contracting, we obtain that for almost every $y\in L$
$$
\lim_{n\rightarrow+\infty}\frac{1}{n}\log {\rm m}\left(Df^n|_{E(y)}\right)\footnote{Here ``{\rm m}" is the mini-norm defined by ${\rm m}(A) = \inf_{\|\nu\|=1} \|A\nu\|=\|A^{-1}\|^{-1}$ for linear isomor-
phism $A$}=\lim_{n\rightarrow+\infty}-\frac{1}{n}\log \|Df^{-n}|_{E(y)}\|\ge -\log \sigma>0.
$$
Therefore, for $\mu$-almost every $x\in L$, $\mu_x$-almost every point has positive Lyapunov exponents along $E$.
According to the ergodic decomposition theorem, with
Lemma~\ref{unstable} and expression $(\ref{decom1})$, one can apply Lemma~\ref{lem:6.2} by letting
$$\mathcal{M}=L,~\eta=\mu|_L,~\mathcal{P}=\mathcal{L},~\eta_x=\mu_x|_L,~k(x)=r(x),~~~\forall x\in L.$$
Consequently, for $\mu|_L$-almost every point $x$, the conditional measures of $\mu_x|_L$ along the disks of $\mathcal{F}$ coincides almost everywhere with the conditional measures of $\mu|_L$ along the same family of disks. Since the conditional measures of $\mu|_L$ are absolutely continuous w.r.t. Lebesgue measures on these unstable disks by assumption, we get that $\mu_x|_{L}$ has absolutely continuous conditional measures. This together with the ergodicity imply that $\mu_x$ is a Gibbs $cu$-state for $\mu$-almost every $x\in L$.

%the conditional measure $\{(\mu_x|_L)_\gamma\}$ is coincide with the conditional measure $\{(\mu|_L)_\gamma\}$ for almost every $\gamma \in \cF$. Note that for almost every $\gamma\in\cL$, $(\mu|_L)_\gamma$ is absolutely continuous with respect to the Lebesgue measure of $\gamma$. Moreover, for $\mu$-almost every $x\in L$, we have
%\begin{itemize}
%\item  $x$ has positive Lyapunov exponents along $E$. This is because that any typical point $x\in L$ is backward $(1,\lambda)$-contracted, thus we obtain
%$$
%\lim_{n\rightarrow+\infty}\frac{1}{n}\log {\rm m}(Df^n|E(x))=\lim_{n\rightarrow+\infty}-\frac{1}{n}\log \|Df^{-n}|E(x)\|\ge -\log \lambda>0.
%$$
%%Therefore, for $\mu$-almost every $x\in L$, $\mu_x$-almost every point has positive Lyapunov exponents along $E$.since $x$ is typical.
%
%\item $\mu_x(L)>0$ by Lemma~\ref{fromMane}.
%\end{itemize}
%
%Since the disintegration of $\mu|L$ is equivalent to Lebesgue for almost every unstable disk $\gamma\in\c L$, one has that $\mu_x$ is a Gibbs $E$-state for $\mu$-almost every $x\in L$.

\end{proof}

\section{Proof of the main results}\label{se6}
In this section, we give the proofs of Theorem \ref{m2}, Theorem \ref{main}, Corollary \ref{cp} and some applications of our results.

\subsection{Proof of Theorem \ref{main}}

Under the settings of Theorem \ref{main}, by Theorem \ref{bac} one knows that $\mu$ has absolutely continuous conditional measures along a measurable partition $\cF$ formed by local unstable disks, all of which are $\sigma$-backward contracting for some $\sigma\in (0,1)$. Consequently, by applying Proposition \ref{pe}, there exist ergodic components of $\mu$ to be Gibbs $cu$-states.

\subsection{Proof of Theorem \ref{m2}}
On the one hand, it follows from Theorem \ref{main} that there exist  ergodic components of $\mu$ to be Gibbs $cu$-states. On the other hand, since $\mu$ has non-positive Lyapunov exponents along $F$, the ergodic decomposition theorem implies that almost every ergodic component of $\mu$ has non-positive Lyapunov exponents along $F$. Altogether one gets the desired result that there are ergodic components of $\mu$ to be SRB measures.

\subsection{Proof of Corollary \ref{cp}}
It can be deduced by Theorem \ref{m2} and absolute continuity of Pesin stable foliation. One can see \cite{y02} for more details.

\subsection{Other results}
By applying Theorem \ref{m2}, we can prove the following result. Note that it has been proved by Cowieson-Young in \cite[Theorem C]{CY05} with a different approach.

\begin{corollaryalph}\label{cy05}
Let $f$ be a $C^2$ diffeomorphism with an attractor $\Lambda$ admitting partially hyperbolic splitting $T_{\Lambda}M=E^u\oplus_{\succ} E^c \oplus_{\succ} E^s$, if ${\rm dim}E^c=1$ then there exist some SRB measures as the ergodic components of any randomly ergodic limit.
\end{corollaryalph}

\begin{proof}
Consider $\mu$ as a randomly ergodic limit, since ${\rm dim}E^c=1$, there eixst following two possibilities: either $\lambda_{E^c}(\mu)\le 0$ or $\lambda_{E^c}(\mu)>0$. In the first case, since every zero noise limit of regular random perturbation of $f$ is a Gibbs $u$-state (e.g. see \cite[Proposition 5]{CY05}), in particular, the randomly ergodic limit $\mu$ is a Gibbs $u$-state, which implies that almost every ergodic components of $\mu$ is a Gibbs $u$-state \cite[Lemma 11.13]{bdv}. Since $\lambda_{E^c}(\mu)\le 0$, by using ergodic decomposition theorem one gets that there are ergodic components of $\mu$ such that they admit non-positive Lyapunov exponents along $E^c$, and thus they are SRB measures. For the later case, one can apply Theorem \ref{m2} by taking $E=E^u\oplus E^c$ and $F=E^s$ to get SRB measures from ergodic components of $\mu$.
\end{proof}

Now we provide a new proof of the existence of SRB/physical measures for systems considered in \cite{MCY17}. Recall that we say $f$ is mostly expanding(resp. contracting) along an invariant sub-bundle $E$ if all the Lyapunov exponents of every Gibbs $u$-state along $E$ are positive(resp. negative).
\begin{corollaryalph}\label{mcy}
Let $f$ be a $C^2$ diffeomorphism with an attractor $\Lambda$ admitting partially hyperbolic splitting $T_{\Lambda}M=E^u\oplus_{\succ} E^{cu} \oplus_{\succ} E^{cs}$ such that $f$ is mostly expanding along $E^{cu}$ and mostly contracting along $E^{cs}$, then there exist physical measures as the ergodic components of any randomly ergodic limit.
\end{corollaryalph}

\begin{proof}
By using the compactness of the set of Gibbs $u$-states, we know that there exists $N\in \NN$ such that
\begin{equation}\label{N}
\int \log \|Df^{-N}|_{E^{cu}}\|{\rm d}\nu>0
\end{equation}
for any Gibbs $u$-state $\nu$ of $f^N$, see \cite[Proposition 3.3]{MCY17} for details. By adding the regular random perturbation of $f^N$, if we take $\mu$ as a randomly ergodic limit, then it is a Gibbs $u$-state, and thus $\mu$ satisfies (\ref{N}). Then applying Theorem \ref{main} we conclude that there is an ergodic component $\eta$ of $\mu$ to be a Gibbs $cu$-state for $f^N$. Since the assumption that $f$ is mostly contracting along $E^{cs}$ implies that $f^N$ is mostly contracting along $E^{cs}$ \cite[Lemma 2.4]{MCY17},  we get that $\eta$ has negative Lyapunov exponents along $E^{cs}$. Hence, the absolute continuity of Pesin stable foliation ensures that $\eta$ is a physical measure for $f^N$. Taking $\xi=\frac{1}{N}\sum_{i=0}^{N-1}f_{\ast}^i\eta$, one can check that $\xi$ is a physical measure for $f$ immediately.
\end{proof}

We emphasize that the partially hyperbolic splittings considered in \cite{MCY17}(also for \cite{AV15}) occurs only on $M$, can not generalized to attractors due to the technique they used there. Moreover,
after showing the existence of physical measures, by more accurate arguments(see the proof of \cite[Theorem 5.1]{MCY17}) one can obtain the finiteness of these measures, i.e., there exist at most finitely many ergodic Gibbs $cu$-states for $f$, which are physical measures.

\appendix
\section{Proof of Lemma \ref{hee}}\label{se7}
\begin{proof}[A sketched proof]
Since $T_{\Lambda}M=E \oplus_{\succ}F$ is dominated, there exists constant $\tau\in (0,1)$ such that
\begin{equation}\label{rd}
\|Df_{\underline{\omega}}^1|_{F(\underline{\omega},x)}\| \cdot \|(Df_{\underline{\omega}}^1|_{E(\underline{\omega},x)})^{-1}\| \le \tau
\end{equation}
for every $(\underline{\omega},x)\in \Lambda_{\varepsilon_0}$. As a consequence of \ref{rd}, for any two subspaces different to random $F$ sub-bundle, the distance of their forward images under $Df_{\underline{\omega}}^n$ are contracted exponentially with rate $\tau$.
Now we take smooth distribution $\widehat{E}$ close to the random sub-bundle $E$, as smoothness implies Lipschitz, one can choose constants $\widehat{C}>0$ such that $\rho(\widehat{E}_x, \widehat{E}_y)\le \widehat{C}d(x,y)$ for any $x,y\in \PP_M(\Lambda_{\varepsilon_0})$. Therefore, we have the following estimations:
%For any distribution $I$, we can define the distance of $I(x)$ and $I(y)$ by using parallel transport along geodesics. More precisely, if we take $I_x(y)$ as the parallel transport of $I(x)$ from $T_x M$ to $T_y M$ along the geodesic connecting $x$ and $y$, then define
%$$
%d(I(x), I(y))=\angle(I_x(y), I(y)).
%$$
%Now we take smooth distribution $\widehat{E}$ close to the random sub-bundle $E$, by domination one knows that there exists $\tau\in (0,1)$ such that for each $(\underline{\omega},x)\in \Lambda_{\varepsilon_0}$, one has
%$$
%Df_{\sigma^{-n}\underline{\omega}}^n\widehat{E}(G^{-n}(\underline{\omega},x))  \to E(\underline{\omega},x)~~\textrm{as}~n\to +\infty.
%$$
%We transport parallelly the space
%For each $n\in \NN$ and $y,z\in \gamma$, the domination implies that
%%\begin{eqnarray*}
%%\rho(Df_{\sigma^{-n}\underline{\omega}}^n\widehat{E}(G^{-n}(\underline{\omega},y)),Df_{\sigma^{-n}\underline{\omega}}^n\widehat{E}(G^{-n}(\underline{\omega},z)))&\le& \tau^n \rho(\widehat{E}(G^{-n}(\underline{\omega},y)), \widehat{E}(G^{-n}(\underline{\omega},y)))\\
%%&\le& \tau^n \rho(\widehat{E}(G^{-n}(\underline{\omega},y)), \widehat{E}(G^{-n}(\underline{\omega},y)))\\
%%&\le &\tau^n \widehat{C}d(f_{\underline{\omega}}^{-n}y,f_{\underline{\omega}}^{-n}y)
%%\end{eqnarray*}
\begin{align*}
&\quad \rho\left(Df_{\sigma^{-n}\underline{\omega}}^n\widehat{E}(G^{-n}(\underline{\omega},y)),Df_{\sigma^{-n}\underline{\omega}}^n\widehat{E}(G^{-n}(\underline{\omega},z))\right)\\
&\quad \quad \quad \quad \quad\quad \le \tau^n \rho\left(\widehat{E}(G^{-n}(\underline{\omega},y)), \widehat{E}(G^{-n}(\underline{\omega},z))\right)\\
& \quad \quad\quad\quad\quad\quad \le \tau^n \widehat{C}d\left(f_{\underline{\omega}}^{-n}(y),f_{\underline{\omega}}^{-n}(z)\right)\\
& \quad \quad\quad\quad\quad\quad \le \tau^n \widehat{C}C\lambda^n d(y,z)\\
& \quad \quad\quad\quad\quad\quad \le \widehat{C}C(\tau \lambda)^n d(y,z)\\
& \quad \quad\quad\quad\quad\quad \le \widehat{C}Cd(y,z).
\end{align*}
Moreover, by property (\ref{rd}) one knows that $Df_{\sigma^{-n}\underline{\omega}}^n\widehat{E}(G^{-n}(\underline{\omega},y))$ converges to $E(\underline{\omega},y)$, and $Df_{\sigma^{-n}\underline{\omega}}^n\widehat{E}(G^{-n}(\underline{\omega},z))$ converges to $E(\underline{\omega},z)$. Hence, by taking $n\to +\infty$ to the left hand of the above estimation we get
$$
\rho(T_y\gamma,T_z\gamma)=
\rho (E(\underline{\omega},y), E(\underline{\omega},z))\le C\widehat{C}d(x,y).
$$
Thus, we complete the proof of this lemma as long as we put $L_0=C\widehat{C}$.
\end{proof}

\vskip 5pt

\noindent Zeya Mi

\noindent School of Mathematics and Statistics

\noindent Nanjing University of Information Science and Technology, Nanjing 210044, China

\noindent mizeya@163.com

%\vskip 5pt
%
%\noindent Dawei Yang
%
%\noindent School of Mathematical Sciences
%
%\noindent Soochow University, Suzhou, 215006, P.R. China
%
%
%\noindent yangdw1981@gmail.com, yangdw@suda.edu.cn

\end{document}